\documentclass[a4paper]{article}

\usepackage{amssymb}
\usepackage{stmaryrd}
\usepackage[all]{xy}
\usepackage{epsf}
\usepackage{enumerate}
\usepackage[pdftex]{graphicx}
\usepackage{amsmath,amscd}
\usepackage{amsthm}
\usepackage{amsopn}
\usepackage{amsfonts}
\usepackage{amssymb}
\usepackage{layout}
\usepackage{verbatim}
\usepackage{alltt}

\usepackage{color}
\usepackage[all]{xy}
\usepackage[latin2]{inputenc}
\usepackage[titles]{tocloft}
\usepackage{titlesec}
\usepackage{titletoc}
\usepackage{abstract}
\usepackage{array}
\usepackage{stmaryrd}
\usepackage{ascii}
\usepackage{mathbbol}
\usepackage{ulsy}

\newtheorem{thm}{Theorem}[section]
\newtheorem{prop}[thm]{Proposition}
\newtheorem{lema}[thm]{Lemma}

\newtheorem{defi}[thm]{Definition}
\newtheorem{ex}[thm]{Example}
\theoremstyle{remark}
\newtheorem{remark}{Remark}

\usepackage{amsmath,amscd}
\usepackage{wrapfig}

\allowdisplaybreaks

\begin{document}

\title{Irreducible Highest Weight Representations Of The Simple n-Lie Algebra}
\author{Dana B\v alibanu, Johan van de Leur\\
\ 
\\
Department of Mathematics,\\
Utrecht University,\\
P.O. Box 80.010,\\
3508 TA Utrecht,\\
The Netherlands\\
\ 
\\
email: D.M.Balibanu@uu.nl,\ 
%\author{Johan van de Leur}
%\address{Depart. of Math., Utrecht University, 3508 TA Utrecht,
%The Netherlands}
J.W.vandeLeur@uu.nl}
\date{}
\maketitle

\begin{abstract}
A. Dzhumadil'daev classified all irreducible finite dimensional
representations of the simple $n$-Lie algebra. Using a slightly
different approach, we obtain in this paper a complete
classification of all irreducible, highest weight modules, including
the infinite-dimensional ones. As a corollary we find all primitive
ideals of the universal enveloping algebra of this simple $n$-Lie
algebra.
\end{abstract}

%%% ----------------------------------------------------------------------

%%% ----------------------------------------------------------------------

%\setcounter{tocdepth}{0}
%\tableofcontents

%\tableofcontents

\section{Introduction}
\label{j8} \noindent In 1985 Filippov \cite{Filippov} introduced a
generalization of a Lie algebra, which he called an $n$-Lie algebra.
The Lie product is taken between $n$ elements of the algebra instead
of two. This new bracket is $n$-linear, anti-symmetric and satisfies
a generalization of the Jacobi identity. For $n=3$ this product is a
special case of the Nambu bracket, well known in physics, which was
introduced by Nambu \cite{N} in 1973, as a generalization of the
Poisson bracket in Hamiltonian mechanics. In recent years this type
of Lie algebra has appeared in the physics literature. For example a
metric 3-Lie algebra is used in the Lagrangian description of a
certain 2+1 dimensional field theory, the so called
Bagger-Lambert-Gustavsson theory \cite{BL}-\cite{G3}. This has lead
to a dramatic increase of interest in $n$-Lie algebras in recent
years.\\
\noindent Let $n$ be a natural number greater or equal to 3. An
$n$-Lie algebra is a natural generalization of a Lie algebra.
Namely:
\begin{defi}A vector space $V$ together with a multi-linear, antisymmetric $n$-ary operation $[,]:\wedge^nV\to
V$ is called an $n$-Lie algebra, $n\geq 3$, if the $n$-ary bracket
is a derivation with respect to itself, i.e,
\begin{equation}
\label{j1}
[[x_1,\ldots,x_n],x_{n+1},\ldots,x_{2n-1}]=\sum_{i=1}^n[x_1,\ldots,[x_i,x_{n+1},\ldots,x_{2n-1}],\ldots,x_n],
\end{equation}
where $x_1,\ldots,x_{2n-1}\in V$.
\end{defi}
\noindent Equation (\ref{j1}) is called \emph{the generalized Jacobi
Identity}.
 \noindent We will only consider $n$-Lie algebras over the
field of complex numbers. \\
\\
\noindent In \cite{Ling} W. Ling has shown that for every $n\geq 3$
there is, up to isomorphism, only one simple $n$-Lie algebra, namely
$\mathbb{C}^{n+1}$ and the operation being given by the generalized
vector product. Denote by $e_1,\ldots,e_{n+1}$ the standard basis of
$\mathbb{C}^{n+1}$, then the $n$-ary bracket is given by:
\[[e_1,\ldots,\hat{e_i},\ldots,e_{n+1}]=(-1)^{n+i+1}e_i,\] where $i$
ranges from 1 to $n+1$ and the hat means that $e_i$ does not
appear in the bracket.\\
\\
 \noindent
A. Dzhumadil'daev studied in \cite{Dzhumadil'daev} the
finite-dimensional, irreducible representation of the simple $n$-Lie
algebra. It is our aim to classify in the present paper both finite-
and infinite-dimensional,
irreducible, highest weight representations of this algebra. The annihilators of these modules are the primitive ideals of the universal enveloping algebra of the simple $n$-Lie algebra.\\
\\
\noindent To every $n$-Lie algebra, one can associate a \emph{ basic
Lie algebra}. For the simple $n$-Lie algebra, denoted throughout
this paper by $A$, the associated basic Lie algebra is $so(n+1)$.
Finding irreducible, highest weight representations of $A$ is
equivalent to finding irreducible highest weight representations of
$so(n+1)$ on which some two-sided ideal, of the universal enveloping
algebra of $so(n+1)$, acts trivially. The method used in this paper
has the advantage of recovering both finite- and
infinite-dimensional, irreducible, highest weight representations.
The annihilators of these modules will supply the primitive ideals
we are searching for. Below, we give the statement of the main
theorem we obtained. The two statements are similar in
nature, the only difference being given by the parity of $n+1$.\\
\\
\noindent Denote by $Z(\lambda)$ the unique irreducible quotient of
the  Verma module of $so(n+1)$ which has highest weight $\lambda$.

\begin{thm}
\label{j6} Let $n\geq 3$, $n+1=2N$ and $t\in\{1,\ldots,N\}$. Denote
by $\pi_1,\ldots,\pi_N$ the fundamental weights of $so(2N)$. Then,
$Z(\lambda)$ is an irreducible representation of the simple $n$-Lie
algebra $A$ if and only if
\begin{displaymath}
\lambda=\left\{
\begin{array}{ll}
x\pi_1 & t=1,\\
(-1-x)\pi_{t-1}+x\pi_t & 1<t<N-1,\\
(-1-x)\pi_{N-2}+x\pi_{N-1}+x\pi_N & t=N-1,\\
(-1-x)\pi_{N-1}+(-1+x)\pi_N & t=N,
\end{array}
\right.
\end{displaymath}
 where $x\in\mathbb{C}.$
\end{thm}
\begin{thm}
\label{j7} Let $n\geq 3$, $n+1=2N+1$ and $t\in\{1,\ldots,N\}$.
Denote the fundamental weights of $so(2N+1)$ by
$\pi_1,\ldots,\pi_N$. Then, $Z(\lambda)$ is an irreducible highest
weight representation of the simple $n$-Lie algebra $A$ if and only
if
\begin{displaymath}
\lambda=\left\{
\begin{array}{ll}
x\pi_1 & t=1,\\
(-1-x)\pi_{t-1}+x\pi_t & 1<t\leq N,
\end{array}
\right.
\end{displaymath}
where $x\in\mathbb{C}.$
\end{thm}

\noindent Chapter 2 gives an introduction to the theory of $n$-Lie
algebras. It presents all the definitions used in this paper.
Chapter 3 describes the strategy adopted for solving the problem. It
also introduces a graphical language which will be very useful
further on. The details of the strategy are presented in Chapter 4.
Here the main results are stated and two ways of arriving to the
same conclusion are presented. The final chapter is concerned with
primitive ideals. The authors would like to express their gratitude
to V. Kac, E. van den Ban and I. M\v arcu\c t for their interest in
this work and for all the good suggestions. 
We  would also like to thank X. Garcia-Martinez, R. Turdibaev, T. van der Linden for pointing out a mistake in a previous version of the article (see  the footnote below).

\section{Main definitions}
\noindent In this section we will give a short introduction to
$n$-Lie algebras. We will concentrate mainly on the most important
definitions and some useful theorems.
\subsection{n-Lie algebras and their basic Lie algebras}

\noindent As mentioned in the introduction, an $n$-Lie algebra is a
vector space $V$, over the field of complex numbers, together with
an $n$-ary operation that satisfies anti-symmetry, multi-linearity
and the generalized Jacobi Identity.
\\
\\
\noindent Let $V$ be an $n$-Lie algebra, $n\geq 3$. We will
associate to $V$ a Lie algebra called \emph{the basic Lie algebra}.
This construction goes as presented below.\\
\\
\noindent Consider $ad:\wedge^{n-1}V\to \textrm{End}(V)$ given by
$ad(a_1\wedge\ldots\wedge a_{n-1})(b)=[a_1,\ldots,a_{n-1},b]$. One
can easily see that we could have chosen the codomain of $ad$ to be
$\textrm{Der}(V)$ (the set of derivations of $V$) instead of
$\textrm{End}(V)$. $ad$ induces a map $\tilde{ad}:\wedge^{n-1}V\to
\textrm{End}(\wedge^\bullet V)$ defined as
$\tilde{ad}(a_1\wedge\ldots\wedge a_{n-1})(b_1\wedge\ldots\wedge
b_m)=\sum_{i=1}^{n-1}b_1\wedge\ldots\wedge[a_1,\ldots,a_{n-1},b_i]\wedge\ldots\wedge
b_m$.\\
\\
\noindent Denote by $\textrm{Inder}(V)$ the set of inner derivations
of $V$, i.e. endomorphisms of the form $ad(a_1,\ldots,a_{n-1})$.\\
\\
\noindent Let $a:=a_1\wedge\ldots\wedge a_{n-1}$ and $b := b_1\wedge
\ldots \wedge b_{n-1} $ be elements of $\wedge^{n-1}V$. Define
$[a,b]= \frac{1}{2}(\tilde{ad}(a) (b)- \tilde{ad}(b)(a))$.
\begin{prop}
Let $V$ be an $n$-Lie algebra such that $\tilde{ad}:\wedge^{n-1}V\to
\textrm{End}(\wedge^\bullet V)$ is skew-symmetric\footnote{
In a previous version of the article this condition was omitted. It was pointed out to us by 
X. Garcia-Martinez, R. Turdibaev, T. van der Linden (see also \cite{Linden}) that this is wrong.
}, then
$[\cdotp,\cdotp]$ defines a Lie algebra structure on $\wedge^{n-1}V$
and $ad:\wedge^{n-1}V\to \textrm{Inder}(V)$ is a surjective Lie
algebra homomorphism.
\end{prop}

\begin{proof}
\noindent The skew-symmetry of the bracket is obvious, and so is the
surjectivity of $ad$, thus we only need to prove that the Jacobi
identity holds. In order to do this, we first show that
$\tilde{ad}(\tilde{ad}(a)(b))= \tilde{ad}(a)\circ
\tilde{ad}(b)-\tilde{ad}(b)\circ \tilde{ad}(a)
=[\tilde{ad}(a),\tilde{ad}(b)]$, and from here it will follow that
$\tilde{ad}$ is a Lie algebra homomorphism, i.e. $\tilde{ad}([
a,b])= \tilde{ad}(a)\circ \tilde{ad}(b)-\tilde{ad}(b)\circ
\tilde{ad}(a) =[\tilde{ad}(a),\tilde{ad}(b)]$. Since both the
left-hand-side and the right-hand-side of the equation are
derivations of the exterior algebra $(\wedge^\bullet V,\wedge)$ (as
previously mentioned) it suffices to show the equality for some
arbitrary $c\in V$.
\begin{align*}
\tilde{ad}(\tilde{ad}&(a)(b))(c)=\tilde{ad}(\tilde{ad}(a_1\wedge\ldots\wedge
a_{n-1})(b_1\wedge\ldots\wedge b_{n-1}))(c)=\\
=&\tilde{ad}(\sum_{i=1}^{n-1}b_1\wedge\ldots\wedge[a_1,\ldots,a_{n-1},b_i]\wedge\ldots\wedge
b_{n-1})(c)=\\
=&\sum_{i=1}^{n-1}[b_1,\ldots,[a_1,\ldots,a_{n-1},b_i],\ldots,
b_{n-1},c]=\\
=&[a_1,\ldots,a_{n-1},[b_1,\ldots,b_{n-1},c]]-[b_1,\ldots,b_{n-1},[a_1,\ldots,a_{n-1},c]]=\\
=&(\tilde{ad}(a)\circ \tilde{ad}(b)-\tilde{ad}(b)\circ
\tilde{ad}(a))(c).
\end{align*}
Hence,
\begin{align*}
\tilde{ad}&([a,b])=\tilde{ad}(\frac{1}{2}(\tilde{ad}(a) (b)-
\tilde{ad}(b)(a)))=\\
=&\frac{1}{2}(\tilde{ad}(\tilde{ad}(a)(b))-\tilde{ad}(\tilde{ad}(b)(a)))=\\
%=&\frac{1}{2}(\tilde{ad}(a)\circ \tilde{ad}(b)-\tilde{ad}(b)\circ \tilde{ad}(a))-\frac{1}{2}(\tilde{ad}(b)\circ \tilde{ad}(a)-\tilde{ad}(a)\circ \tilde{ad}(b))=\\
%=&\frac{1}{2}(2\tilde{ad}(a)\circ \tilde{ad}(b)-2\tilde{ad}(b)\circ \tilde{ad}(a))=\\
=&(\tilde{ad}(a)\circ \tilde{ad}(b)-\tilde{ad}(b)\circ \tilde{ad}(a))=\\
=&[\tilde{ad}(a),\tilde{ad}(b)].
\end{align*}
\noindent Now the Jacobi identity follows easily and we are done.
%\[4([ a,[ b,c]]+[ b,[ c,a]] + [ c,[ a,b]])=\]\[2(\tilde{ad}(a)([
%b,c])-\tilde{ad}([ b,c])(a)+\tilde{ad}(b)([ c,a])-\tilde{ad}([
%c,a])(b)+\]\[\tilde{ad}(c)([ a,b])-\tilde{ad}([
%a,b])(c))=\]\[\tilde{ad}(a)(\tilde{ad}(b)(c)
%-\tilde{ad}(c)(b))-(\tilde{ad}(b)\circ
%\tilde{ad}(c)-\tilde{ad}(c)\circ \tilde{ad}(b))(a)+\]
%\[\tilde{ad}(b)(\tilde{ad}(c)(a)
%-\tilde{ad}(a)(c))-(\tilde{ad}(c)\circ
%\tilde{ad}(a)-\tilde{ad}(a)\circ \tilde{ad}(c))(b)+\]
%\[\tilde{ad}(c)(\tilde{ad}(a)(b)
%-\tilde{ad}(b)(a))-(\tilde{ad}(a)\circ
%\tilde{ad}(b)-\tilde{ad}(b)\circ \tilde{ad}(a))(c)=\]
%\[\tilde{ad}(a)\circ \tilde{ad}(b)(c)
%-\tilde{ad}(a)\circ \tilde{ad}(c)(b)-\tilde{ad}(b)\circ
%\tilde{ad}(c)(a)+\tilde{ad}(c)\circ \tilde{ad}(b)(a)+\]
%\[\tilde{ad}(b)\circ \tilde{ad}(c)(a)
%-\tilde{ad}(b)\circ \tilde{ad}(a)(c)-\tilde{ad}(c)\circ
%\tilde{ad}(a)(b)+\tilde{ad}(a)\circ \tilde{ad}(c)(b)+\]
%\[\tilde{ad}(c)\circ \tilde{ad}(a)(b)
%-\tilde{ad}(c)\circ \tilde{ad}(b)(a)-\tilde{ad}(a)\circ
%\tilde{ad}(b)(c)+\tilde{ad}(b)\circ \tilde{ad}(a)(c)=0\]
\end{proof}

\begin{defi}
A subspace $V'\subseteq V$ is called a subalgebra of the $n$-Lie
algebra $V$ if $[V',\ldots,V']\subseteq V'$. A subalgebra
$I\subseteq V$ of an $n$-Lie algebra $V$ is called an ideal if
$[I,V,\ldots,V]\subseteq I$. An $n$-Lie Algebra $V$ is called simple
if it has no proper ideals besides 0.
\end{defi}
\noindent We will illustrate these two notions by means of the
following example.
\begin{ex}
\noindent Let $V=\{f:\mathbb{R}^n\to \mathbb{R}| f\textrm{ of class
}C^\infty\}$, $n\geq 3$, and define the $n$-ary bracket on $V$ as
the Jacobian.
\[[f_1,\ldots,f_n]=\left|
\begin{array}{ccc}
\frac{\partial f_1}{\partial x_1}&\ldots&\frac{\partial
f_n}{\partial x_1}\\
\vdots&\ddots&\vdots\\
\frac{\partial f_1}{\partial x_n}&\ldots&\frac{\partial
f_n}{\partial x_n}\\
\end{array}
\right|.\] \noindent $V$ together with the operation given by the
Jacobian forms an $n$-Lie algebra. Denote by $V'$ the subset of $V$
containing the polynomial functions in $V$. Then $V'$ together with
the inherited operation forms a subalgebra of the $n$-Lie algebra
$V$. An ideal of this $n$-Lie algebra can be obtained as follows.
Denote by $I$ the set of functions in $V$ which are flat at the
origin. Then $I$ is an ideal of $V$.
\end{ex}

\noindent In his PhD-thesis in 1993 \cite{Ling}, Wuxue Ling
classified all simple $n$-Lie algebras, $n\geq 3$. He arrived at the
following conclusion: for every $n\geq 3$ the only simple $n$-Lie
algebra is isomorphic to $\mathbb{C}^{n+1}$, where the bracket is
defined as follows. Denote by $e_1,\ldots,e_{n+1}$ the canonical
basis of $\mathbb{C}^{n+1}$, then:
\[[e_1,\ldots,\hat{e_i},\ldots,e_{n+1}]=(-1)^{n+i+1}e_i,\] where
$\hat{e_i}$ means that $e_i$ is omitted.
\begin{remark} From now on we will denote a generic $n$-Lie algebra
by $V$ and the simple one by $A$. The vector space $A$ comes
together with an inner product, $\langle e_i,e_j\rangle =\delta_{i,j}$, and the
standard orientation form, $e_1\wedge\ldots\wedge e_{n+1}$. Hence,
we can observe that the simple $n$-Lie Algebra $A$ is just the
vector space $\mathbb{C}^{n+1}$ together with the Hodge star
operation $*:\wedge^{n}\mathbb{C}^{n+1}\to\mathbb{C}^{n+1} $.
Moreover $*:\wedge^2A\to\wedge^{n-1}A$ is an isomorphism of Lie
algebras; $\wedge^{n-1}A\simeq\wedge^2A\simeq so(n+1)$.
\end{remark}
%\noindent In the case of the simple n-Lie algebra $A$,
%$\textrm{Der}(A)=\textrm{Inder}(A)$ is a semisimple Lie Algebra,
%where $\textrm{Inder}(A)$, as already mentioned, denotes the set of
%inner derivations of $A$. Thus, we have obtained an extension of a
%semisimple Lie algebra $\wedge^{n-1}A/ \textrm{Ker}(ad)\simeq
%\textrm{Inder}(A)$ by an abelian ideal $\textrm{Ker}(ad)$.
%\[0\longrightarrow \textrm{Ker}(ad)\stackrel{i}{\longrightarrow}\wedge^{n-1}A\stackrel{ad}{\longrightarrow}\textrm{Inder}(A)\longrightarrow 0.\]
%\noindent The semisimplicity of $\textrm{Inder}(A)$ implies now that
%$H^2(\textrm{Inder}(A),\textrm{Ker}(ad))=0$, and hence
%$\wedge^{n-1}A=\textrm{Inder}(A)\ltimes
%\textrm{Ker}(ad)$, where $\ltimes$ denotes the semidirect product.\\
%\\

\subsection{n-Lie modules}

In this section we assume that $V$ is an $n$-Lie algebra such that the map  $\tilde{ad}:\wedge^{n-1}V\to
\textrm{End}(\wedge^\bullet V)$ is skew-symmetric.
\begin{defi}
A vector space $M$ is called an $n$-Lie module for the $n$-Lie
algebra $V$, if on the direct sum $V\oplus M$ there is the structure
of an $n$-Lie algebra, such that the following conditions are
satisfied:
\begin{itemize}
\item $V$ is a subalgebra;
\item $M$ is an abelian ideal, i.e. when at least two slots of the
$n$-bracket are occupied by elements in $M$ the result is 0.
\end{itemize}
\end{defi}

\noindent Let $M$ be an $n$-Lie module of $V$. Then, $M$ is a module
for the basic Lie algebra  $\wedge^{n-1}V$, where the action is
given by $x_1\wedge\ldots\wedge x_{n-1}.m=[x_1,\ldots,x_{n-1},m].$
On the direct sum of the two spaces $M$ and $V$ we must have the
structure of an $n$-Lie algebra. This means that the generalized
Jacobi identity has to hold. If we write out this condition, we
obtain a two sided ideal of the universal enveloping algebra of
$\wedge^{n-1}V$, generated by the elements
\[x_{a_1,\ldots,a_{2n-2}}=[a_1,\ldots,a_n]\wedge
a_{n+1}\wedge\ldots\wedge a_{2n-2}-\]\[
-\sum_{i=1}^n(-1)^{i+n}(a_1\wedge\ldots\wedge\hat{a_i}\wedge\ldots\wedge
a_n)(a_i\wedge a_{n+1}\wedge\ldots\wedge a_{2n-2}).\] This two sided
ideal must act trivially on $M$. We will denote this ideal by $Q(V)$
and the universal enveloping algebra of the basic Lie algebra by
$U:=U(\wedge^{n-1}V)$. On the other hand, any module of the basic
Lie algebra, on which $Q(V)$ acts trivially, satisfies the
conditions defining an $n$-Lie module of $V$. Thus, because of this
property, it makes sense to define:
\begin{defi}
The universal enveloping algebra of the $n$-Lie algebra $V$, denoted
by $U(V)$, is defined as $U/Q(V)$.
\end{defi}
\noindent Because of this definition, we obtain the same relation,
between the representations of an $n$-Lie algebra and the
representations of its universal enveloping algebra, as for Lie
algebras. Namely, representations of the $n$-Lie algebra $V$ are in
$1-1$ correspondence with the representations of the universal
enveloping algebra $U(V)$. It was proven in \cite{Dzhumadil'daev}
that an $n$-Lie module $M$ of $V$ is irreducible / completely
reducible if and only if it is irreducible
/ completely reducible as a module of the Lie algebra $\wedge^{n-1} V$.\\
\\
\noindent For Lie algebras, \emph{primitive ideals} are defined to
be two sided ideals of the universal enveloping algebra, which are
annihilators of irreducible representations. It makes sense to
define them the same in the case of $n$-Lie algebras.
\begin{defi} Let $V$ be an $n$-Lie algebra and $U(V)$ its universal
enveloping algebra. A two-sided ideal of $U(V)$ is called primitive
if it is the annihilator of some irreducible module of $V$.
\end{defi}
\noindent Let $M$ be an irreducible module of the $n$-Lie algebra
$V$. The annihilator of this module, $\textrm{Ann}(M)$, is a
primitive ideal of $U(V)$, by definition. Since $U(V)=U/Q(V)$, we
may conclude that $\textrm{Ann}(M)$ is a primitive ideal of $U$
which includes $Q(V)$. On the other hand, let $I$ be a primitive
ideal of $U$ which includes $Q(V)$. Then, there exist an irreducible
module of $U$, call it $M$, which is annihilated by $I$, and thus by
$Q(V)$. This transforms $M$ into a module of $U(V)$ and $I$ into a
primitive ideal of $U(V)$. Hence, we may conclude that primitive
ideals of $U(V)$ are in $1-1$ correspondence with
the primitive ideals of $U$ which contain $Q(V)$.  \\
\\
\noindent Primitive ideals of $U$ are, by definition, annihilators
of the irreducible modules. It was proven in \cite{Duflo}, by M.
Duflo, that one does not need to look at all irreducible modules.
Namely, it suffices to look only at the highest weight, irreducible
ones and these will supply all the primitive ideals. Irreducible
modules of an $n$-Lie algebra are in particular irreducible modules
of the associated basic Lie algebra. We define highest weight
modules of the $n$-Lie algebra similarly. A module of the $n$-Lie
algebra $V$ is called a highest weight module, if it is a highest
weight module of the basic Lie algebra $\wedge^{n-1}V$. This means
that our problem of determining the primitive ideals of $U(A)$, can
be reformulated as determining the highest weight, irreducible
representations of $A$.

\subsection{The  simple n-Lie algebra A}

\noindent Recall that we denoted by $A$ the simple,
$n+1$-dimensional $n$-Lie algebra.
\begin{remark}
Note that the simple $n$-Lie algebra $A$ automatically satisfies the condition that the map  $\tilde{ad}:\wedge^{n-1}A\to
\textrm{End}(\wedge^\bullet A)$ is skew-symmetric.
\end{remark}
\noindent
The ideal $Q(A)$ is generated by the elements
\[x_{a_1,\ldots,a_{2n-2}}=[a_1,\ldots,a_n]\wedge
a_{n+1}\wedge\ldots\wedge a_{2n-2}-\]\[
-\sum_{i=1}^n(-1)^{i+n}(a_1\wedge\ldots\wedge\hat{a_i}\wedge\ldots\wedge
a_n)(a_i\wedge a_{n+1}\wedge\ldots\wedge a_{2n-2}).\] For the simple
$n$-Lie algebra $A$ we can compute these generators further. Instead
of the elements $a_i\in A$ we plug in elements of the standard basis
of $A$ and use the isomorphism of Lie algebras which sends the
element
$e_1\wedge\ldots\wedge\hat{e_i}\wedge\ldots\wedge\hat{e_j}\wedge\ldots\wedge
e_{n+1}\in \wedge^{n-1}A$ to
\begin{equation}
\label{j2} e^{ij}:=e_i\wedge e_j= \frac{e_i\otimes e_j-e_j\otimes
e_i}{2}=E_{ij}-E_{ji}\in so(n+1),\quad 1\leq i<j\leq n+1.
\end{equation}
We obtain the following relations:
\begin{displaymath}
x_{i,k,l,m} = \left\{
\begin{array}{ll}
    e^{ik}e^{lm}-e^{il}e^{km}+e^{im}e^{kl} & \textrm{if }i,k,l,m\textrm{ are all distinct,} \\
    0 & \textrm{otherwise,}
\end{array}
\right.
\end{displaymath}
where $i<k<l<m\in\{1,\ldots,n+1\}$. Here, $x_{i,k,l,m}$ is just a
short-hand notation of the generator denoted before by
$x_{a_1,\ldots,a_{2n-2}}$. For the detailed computation, we refer
the reader to \cite{Dzhumadil'daev}.
\begin{remark} We will show later on, that the condition on the indices can be
dropped. Although it is not necessary, we will still assume them to
be ordered.
\end{remark}
\noindent One can easily see that these relations live in $S^2(so(n+1)).\footnote{This will play an important role in the near future.}$\\
\\
\noindent Denote by $R:=\textrm{Span}\{x_{j,k,l,m}\}_{1\leq j,k,l,m
\leq n+1 }$ and recall that $\wedge^{n-1}A\simeq\wedge^2A$.

\begin{lema} $R$ is a finite dimensional $\wedge^{2}A$-module.
\end{lema}
\begin{proof} An easy computation.
\end{proof}
\begin{lema}
The left-sided ideal generated by $R$ equals to the right-sided
ideal generated by $R$ and equals $Q(A)$.
\end{lema}
\noindent Consider the $\wedge^2A$-module $\wedge^4A$ and the map
$\psi:\wedge^4A\to S^2(\wedge^2A)$ defined on monomials
as:\[v_i\wedge v_j\wedge v_k\wedge v_l\mapsto (v_i\wedge
v_j)\odot(v_k\wedge v_l)-(v_i\wedge v_k)\odot(v_j\wedge
v_l)+(v_i\wedge v_l)\odot(v_j\wedge v_k).\footnote{The reader should
observe that these are exactly the generators of $Q(A)$, the only
difference being the notation.}\]
\begin{lema} Let $\tilde{\psi}:\wedge^4A\to R$ be the restriction of
the map $\psi$ to $R\subset S^2(\wedge^2A)$. Then, $\tilde{\psi}$ is
an isomorphism of Lie modules.
\end{lema}
\noindent Of course, we can also define the map
$\phi:S^2(\wedge^2A)\to \wedge^4A$ defined on monomials as:
\[(v_i\wedge v_j)\odot(v_k\wedge v_l)\mapsto v_i\wedge v_j\wedge v_k\wedge v_l.\]
\noindent Observe that $\phi\circ\psi=3Id$. Then, $\textrm{Ker}\phi$
is a subrepresentation of $\wedge^2A$ in $S^2(\wedge^2A)$,
complementary to $R$.\footnote{A rigorous proof of this can be found
later on. } Hence, we have obtained the following decomposition into
submodules:
\[S^2(\wedge^2A)=\psi(\wedge^4A)\oplus \textrm{Ker}\phi.\]
\noindent On the other hand, we already know that
\[U(so(n+1))=\bigoplus_k S^k(\wedge^2A).\] Thus, the universal enveloping algebra of the simple
$n$-Lie algebra $A$ is $U=\bigoplus_k
S^k(\wedge^2A)/Q(A)$.\footnote{This fact will supply the proof to a
future statement about the PBW-basis of $U$.}

\section{Finding irreducible highest weight representation of A}

\noindent Our main concern in this section will be to find
irreducible, highest weight representations of the simple $n$-Lie
algebra. Once this is done, it will be easy, by looking at the
highest weight, to figure out which of these representations are
finite dimensional and which are not.

\subsection{A graphical interpretation of the generators of Q(A)}

\noindent In order to understand the relations which generate $Q(A)$
better, we want to view them as graphical diagrams. On the basis
elements of $so(n+1)$ we define the lexicographical order, namely
$e^{i_1j_1}\leq e^{i_2j_2}\iff i_1<i_2,\textrm{ or }i_1=i_2\textrm{
and }j_1\leq j_2$. In this case, we say that
$(i_1,j_1)\leq(i_2,j_2)$. We always assume that $i<j$. If this is
not the case, we can interchange them by the following rule:
$e^{ij}=-e^{ji}$. Taking into account that two basis elements of
$so(n+1)$, which are not in lexicographical order, can be reordered,
using the Lie bracket, at the expense of some term of degree one
less, it becomes easy to give a PBW-basis of $U$, the universal
enveloping algebra of the basic Lie algebra $\wedge^{n-1}A\simeq
so(n+1)$. Let $U_k=\{e\in U|e=e^{i_1j_1}e^{i_2j_2}\ldots e^{i_kj_k},
\textrm{ where }(i_1,j_1)\leq(i_2,j_2)\leq\ldots\leq(i_k,j_k)\}$,
i.e. all simple elements of degree $k$. Then a
Poincare-Birkhoff-Witt basis is given by $\bigcup_k U_k$.\\
\\
\noindent For any simple element of degree 1, i.e. some $e^{ij}$, we
can represent it graphically as $n+1$ ordered points with an
oriented arc going from the $i$'th point to the $j$'th. (Recall that both $i$ and $j$ range from 1 to $n+1$.)\\
\includegraphics[width=0.15\textwidth]{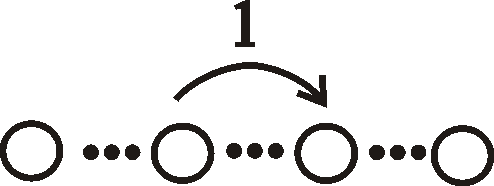}\\

\noindent Changing the orientation of this arc is the same as
changing the order of $i$ and $j$, thus it results in a minus sign
in front of the diagram. Some arbitrary product in $U$ can be
represented similarly as $n+1$ ordered points with arcs connecting
them. Multiple arcs between the same two points are allowed, each of
this arcs having its own number above it. This number stands for the
place the basis element occupies in the product. Multiplication of
such elements can be translated, in this graphical language, as the
overlapping of such diagrams, where the numbers above the arcs in
the second diagram have to be shifted by a number
equal to the number of arcs in the first diagram.\\
\\
\noindent In $U$ the following commutation relations hold. These are
the relations mentioned in the beginning of this section,
represented now graphically. \\

\noindent \includegraphics[width=0.3\textwidth]{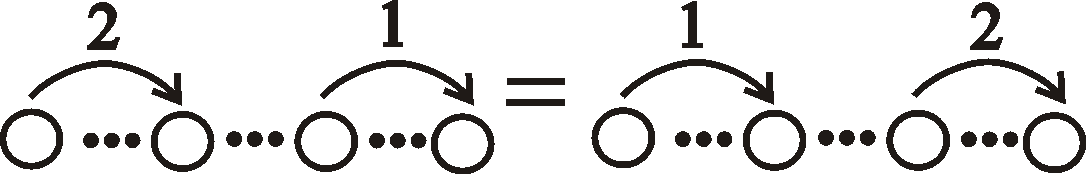}\\
\includegraphics[width=0.3\textwidth]{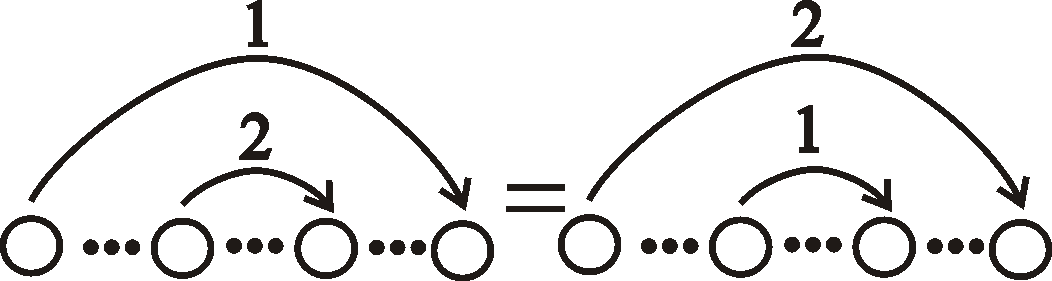}\\
\includegraphics[width=0.35\textwidth]{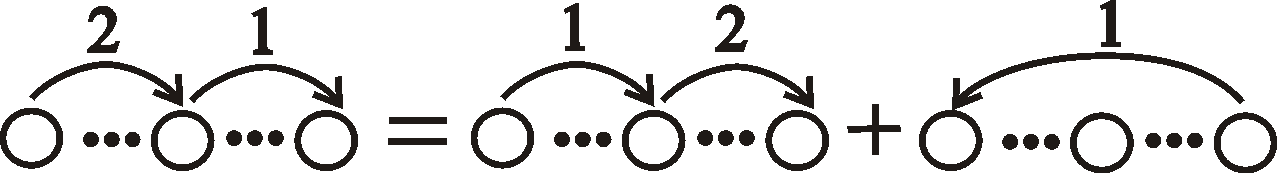}\\
\includegraphics[width=0.35\textwidth]{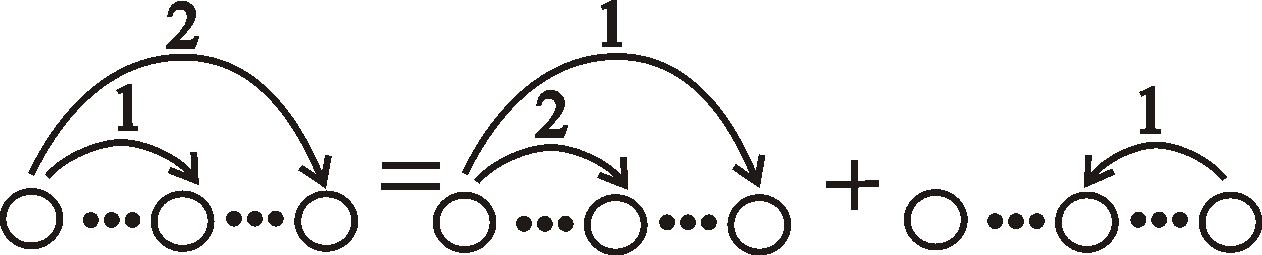}\\
\includegraphics[width=0.35\textwidth]{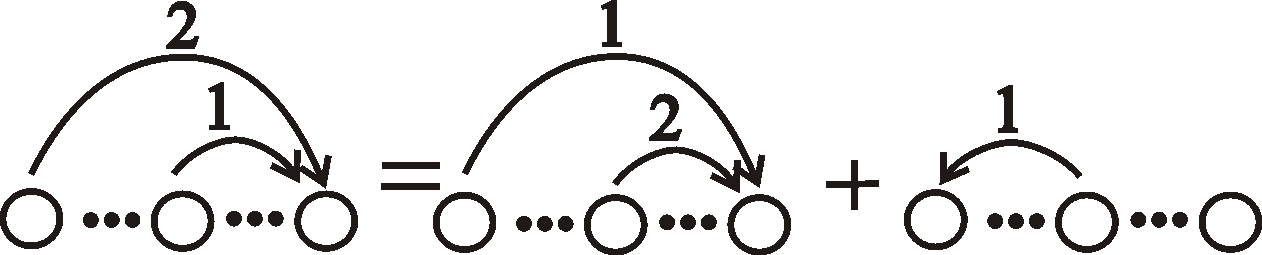}\\

\noindent In view of these relations, the numbers above the arcs can
be dropped and the order of the arrows can be assumed to be the
lexicographical order.\\
\\
\noindent Now we can give a graphical interpretation to the
generators $x_{i,j,k,l}$ of $Q(A)$. They tell us that we can resolve
intersections in any diagram:\\

\includegraphics[width=0.47\textwidth]{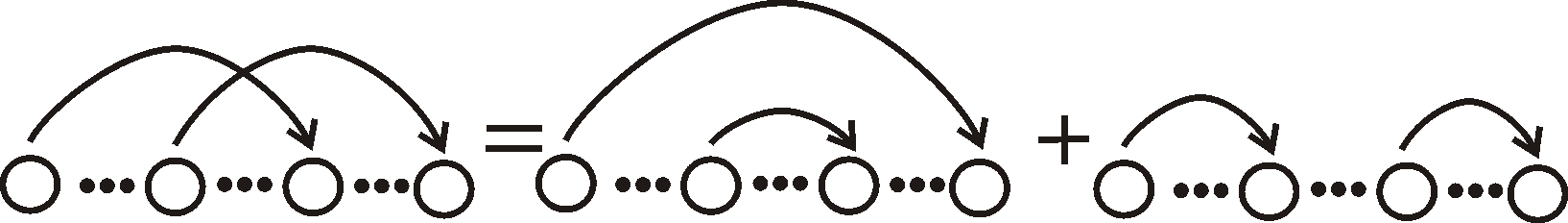}.\\

\noindent From here we can conclude that a PBW-basis for the
universal enveloping algebra of $A$ can be obtained from the
PBW-basis of $U(\wedge^{2}A)$ by dropping those elements
which in their graphical representation have intersecting arcs.\\
\\
\noindent Although this graphical interpretation does not supply
much insight just yet, it will become very useful further on.

\subsection{First steps towards a solution}

\noindent The main algebraic object we will work with is the simple
Lie algebra $so(n+1)$ (respectively the semi-simple one, in the case
of $so(4)$). Its universal enveloping algebra will be denoted as
before by $U$. Let $H$ be a Cartan subalgebra of $so(n+1)$ and
$\Phi$ the corresponding root system. We denote by $\Phi^+$ the set
of positive roots of $\Phi$ and by $\Phi^-$ the set of negative
ones. Fix $\lambda$ in $H^*$. Consider the left ideal of $U$,
$I(\lambda)$, generated by all $x_{\alpha},\textrm{ with }\alpha\in
\Phi^+$ and all $h-\lambda(h)1$, where $h\in H$. Then
$V(\lambda):=U/I(\lambda)$ is a highest weight module of $so(n+1)$
with highest weight $\lambda$, called the Verma module of weight
$\lambda$. $V(\lambda)$ might not be irreducible but it has a unique
irreducible quotient. Define $Z(\lambda)$ to be $U/J(\lambda)$,
where $J(\lambda)$ is the unique maximal left ideal of $U$
containing the left ideal $I(\lambda)$. Then $Z(\lambda)$ is an
irreducible highest weight module with highest weight $\lambda$. Our
goal is to determine for which $\lambda\in H^*$, $Z(\lambda)$ is an
irreducible module of the $n$-Lie algebra $A$, i.e. for which $\lambda\in H^*$, the two-sided ideal $Q(A)$ acts trivially on $Z(\lambda)$. \\
\\
\noindent Before we do this, we need to show that our results will
be independent of the choice of the Cartan subalgebra, i.e. for any
two Cartan subalgebras $H$ and $H'$ of $so(n+1)$, the sets of
weights corresponding to these subalgebras, such that $Z(\lambda)$
is an irreducible representation of the $n$-Lie algebra $A$,
are related by an isomorphism.\\
\\
\noindent Let $H$ and $H'$ be two Cartan subalgebras of $so(n+1)$,
$\lambda\in H^*$, $\lambda'\in (H')^*$ and $I(H,\lambda)$,
$I(H',\lambda')$ the two left ideals corresponding to each of these
Cartan subalgebras (the notation now keeps track of the Cartan
subalgebra as well). Denote by $J(H,\lambda)$, $J(H',\lambda')$ the
maximal left ideals of $U$ which include the two ideals above. As
will be clear later on, the problem of $Z(\lambda)$, respectively
$Z(\lambda')$, being an irreducible $A$-module translates as
$Q(A)\subseteq J(H,\lambda)$, respectively $Q(A)\subseteq
J(H',\lambda')$.\\
\\
\noindent Let \[\Lambda:=\{\alpha\in H^*|Q(A)\subseteq
J(H,\alpha)\},\]
\[\Lambda':=\{\alpha'\in (H')^*|Q(A)\subseteq J(H',\alpha')\}\]
and $\varphi:so(n+1)\to so(n+1)$ be a Lie algebra isomorphism for
which $\varphi(H)=H'$. Then $\varphi$ induces an
automorphism $\tilde{\varphi}:U\to U$.\\
\\
\noindent We want the following equality to
hold:\[(\varphi_{|H})^*(\Lambda')=\Lambda,\] which is the same thing
as \[Q(A)\subseteq J(H,\lambda)\textrm{ if and only if }
Q(A)\subseteq J(H',\lambda'),\]  for
$\lambda'=(\varphi_{|H})^*(\lambda)$. Since $\varphi$ is an
isomorphism, it is enough to show just one implication of the above
equivalence. Assume that $Q(A)\subseteq J(H,\lambda)$. Then,
\[Q(A)\subseteq J(H,\lambda)\Rightarrow
\tilde{\varphi}(Q(A))\subseteq\tilde{\varphi}(J(H,\lambda))=J(H',\lambda').\]
Hence, if we can show that $\tilde{\varphi}(Q(A))=Q(A)$ we are
done.\\
\\
\noindent We will prove this by finding a canonical description for
the two sided ideal $Q(A)$. It will then follow that any
automorphism of $U$, coming from a Lie algebra isomorphism, will map
$Q(A)$ into itself.

\subsubsection{The canonical description of Q(A)}

\noindent As explained above, finding a canonical description for
$R$, the $\textrm{Span}$ of the relations, is a crucial step in
proving that our strategy is independent of the Cartan subalgebra we
are working with. For this purpose, we will explain how $R$ appears
as an eigenspace of the Casimir acting on
$so(n+1)\otimes so(n+1)$.\\
\\
\noindent Denote by $c$ the universal Casimir element, given by the
formula\[c=-\frac{1}{2n}\sum_{i<j}(e^{ij})^2.\]

\noindent Define $\bar{c}=\frac{n}{2}c-n\cdot Id$. Our goal is to
show that $R=\textrm{Eigenspace}_{\bar{c}}(-2)$, i.e. the eigenspace
of $\bar{c}$ with
eigenvalue -2.\\
\\
\noindent Note that
\[so(n+1)\otimes so(n+1)=\wedge^2 so(n+1)\oplus
S^2(so(n+1)),\] where \[x\otimes y=\frac{x\otimes y-y\otimes
x}{2}+\frac{x\otimes y+y\otimes x}{2}.\] We denote the first term of
the sum on the right-hand-side by $x\wedge y$ and the second one by
$x\odot y$. It is easy to see that $\forall r\in R: \bar{c}\cdot
r=-2r$. This shows that $R\subseteq
\textrm{Eigenspace}_{\bar{c}}(-2)$. So, it remains to show that the
converse inclusion also holds. We will prove this by explicitly
exhibiting a basis of $S^2(so(n+1))$ composed of eigenvectors of
$\bar{c}$.\\
 \\
 \noindent To do this we will use the following formulas: \[
 \bar{c}(e^{ab}\odot e^{cd})=
 \left \{
 \begin{array}{ll}
 -\frac{1}{2}\sum_{i=1}^{n+1}(e^{ai}\odot e^{ai}+e^{bi}\odot
 e^{bi})+e^{ab}\odot e^{ab} & a=c, b=d,\\
-\frac{1}{2}\sum_{i=1}^{n+1}e^{ib}\odot e^{id}+e^{ab}\odot e^{ad} & a=c, b\neq d,\\
-\frac{1}{2}\sum_{i=1}^{n+1}e^{ai}\odot e^{ci}+e^{ab}\odot e^{cb} &
a\neq c, b= d,\\
-\frac{1}{2}\sum_{i=1}^{n+1}e^{ai}\odot e^{id}+e^{ab}\odot e^{bd} &
b=c,\\
-\frac{1}{2}\sum_{i=1}^{n+1}e^{ib}\odot e^{ci}+e^{ab}\odot e^{ca} &
a=d,\\
-e^{ad}\odot e^{bc}+e^{bd}\odot e^{ac} &\text{otherwise,}
 \end{array}
 \right .
 \]
where $a<b$ and $c<d$.\\
\\
\noindent Denote by \[B:=<e^{ab}\odot e^{cd}|
a<b,c<d\in\{1,\ldots,n+1\};|\{a,b,c,d\}|=4>\]
\[C:=<e^{ab}\odot e^{cd}|a<b,c<d\in\{1,\ldots,n+1\};|\{a,b,c,d\}|=3>\] and by \[D:=<e^{ab}\odot e^{ab}|1\leq a<b\leq n+1>.\]
Then, $S^2(so(n+1))=B\oplus C\oplus D$ and
$\textrm{dim}(S^2(so(n+1)))=\textrm{dim}(B)+\textrm{dim}( C)+
\textrm{dim}(D)=3\binom{n+1}{4}+3\binom{n+1}{3}+\binom{n+1}{2}$. We will find a basis of eigenvectors for each of the three subspaces $B$, $C$ and $D$.\\
\\
\noindent For fixed $a,b,c,d\in \{1,\ldots,n+1\}$, all distinct,
denote by \[V_{abcd}:=<e^{ab}\odot e^{cd},e^{ad}\odot
e^{bc},e^{bd}\odot e^{ac}>,\] the 3-dimensional subspace of $B$
generated by the three vectors. Then, a basis of eigenvectors for
$V_{abcd}$ is given by the 3 vectors \[e^{ab}\odot
e^{cd}+e^{ad}\odot e^{bc}-e^{ac}\odot e^{bd} \in
\textrm{Eigenspace}_{\bar{c}}(-2),\]
\[
e^{ab}\odot e^{cd}+e^{ac}\odot e^{bd}\in
\textrm{Eigenspace}_{\bar{c}}(1),\]
\[e^{ad}\odot e^{bc}+e^{ac}\odot e^{bd} \in
\textrm{Eigenspace}_{\bar{c}}(1).\] Since \[B=\bigoplus_{1\leq
a<b<c<d\leq n+1}V_{abcd},\] we obtain a basis of eigenvectors of
$\bar{c}$ acting on $B$. The reader is urged to note that the
eigenvectors corresponding to the eigenvalue -2 are the relations
$x_{a,b,c,d}$. \\
\\
\noindent For the subspace $C$ the procedure is similar. For
$a<c\in\{1,\ldots,n+1\}$ denote by \[V_{ac}=<e^{ai}\odot
e^{ci}|i\in\{1,\ldots,n+1\}-\{a,c\}>.\] Then \[C=\bigoplus_{1\leq
a<c\leq n+1}V_{ac}.\] We regard $V_{ac}$ as the subspace of $C$
spanned by $n-1$ vectors: $v_k:=e^{ak}\odot e^{ck}, $ for
$k\in\{1,\ldots,n+1\}-\{a,c\}$; and consider $n-2$ eigenvectors of
$\bar{c}$ given by $v_1-v_2, \ldots,v_n-v_{n+1}$. These will be
linearly independent eigenvectors of $\bar{c}$ with eigenvalue 1. To
obtain a basis of $V_{ac}$ we still need one more eigenvector,
namely
\[\sum_{i=1}^{n+1}e^{ai}\odot e^{ci}\in \textrm{Eigenspace}_{\bar{c}}(-\frac{n-1}{2}).\]
By taking the union of these bases, we obtain a basis of $C$ given
by eigenvectors, none of which belongs to the eigenspace of
$\bar{c}$ with eigenvalue -2, for generic $n$. By this we mean that
for $n\neq 5$ none of the eigenvalues obtained so far will be equal
to -2. On the other hand if $n=5$, then
$-\frac{n-1}{2}=-2$. Hence, this case has to be treated separately.\\
\\
\noindent The last subspace we need to take into account is the
subspace $D$:\[D=<e^{ij}\odot e^{ij}|1\leq i<j\leq n+1>.\] We use
the following notation \[X_{a,b}:=e^{ab}\odot e^{ab},\]
\[S_a=\sum_{i=1}^{n+1}e^{ai}\odot e^{ai}.\] With this notation, a
basis of eigenvectors is given by:
\[X_{a,b}-X_{a+1,b}-X_{a,b+1}+X_{a+1,b+1}\in \textrm{Eigenspace}_{\bar{c}}(1),\textrm{ for }1\leq a<b-1\leq n,\]
\[S_a-S_{a+1}\in \textrm{Eigenspace}_{\bar{c}}(-\frac{n-1}{2}),\textrm{ for }1\leq a\leq n,\]
\[\sum_{a=1}^{n+1}S_a\in \textrm{Eigenspace}_{\bar{c}}(-n),\]
\[X_{p-1,p}-X_{p-1,p+1}-X_{p,p+2}+X_{p+1,p+2}\in \textrm{Eigenspace}_{\bar{c}}(1),\textrm{ for }p\in\{2,\ldots,n-1\}.\]
Hence, for $n\neq 5$, we have obtained a basis of eigenvectors for
$S^2(so(n+1))$ and it becomes clear that
$R=\textrm{Eigenspace}_{\bar{c}}(-2)$.\\
\\
\noindent Next we treat the case $n=5$, when the eigenvalues -2 and
$-\frac{n-1}{2}$ coincide. Thus, we
obtain that $R\subsetneq \textrm{Eigenspace}_{\bar{c}}(-2)$.\\
\\
\noindent  If $n=5$ then $so(n+1)=so(6)$ and
$\wedge^{4}\mathbb{C}^{6}$ is an irreducible representation of
$so(6)$ isomorphic to the adjoint representation and isomorphic to
$R$. Let $\pi_i$, $i\in\{1,2,3\}$, be the fundamental weights of
$so(6)$ (as presented in the Dynkin diagram below) and
$V_{(a,b,c)}^d$ the irreducible highest weight $so(6)$-module of
dimension $d$ and with highest weight $a\pi_1+b\pi_2+c\pi_3$. Then,
the decomposition of $S^2 so(6)$ into irreducible modules is given
by:
\[S^2 so(6)=V_{(0,0,0)}^1+V_{(0,1,1)}^{15}+V_{(0,2,2)}^{84}+V_{(2,0,0)}^{20}.\]
It is easy to see that $R$ appears only once in this decomposition
($R\simeq V_{(0,1,1)}^{15}$), hence we can conclude that our results
will be independent of the Cartan subalgebra.

\begin{figure*}[h!]
  \caption{The Dynkin diagram of $so(6)$.}
  \centering
    \includegraphics[width=0.14\textwidth]{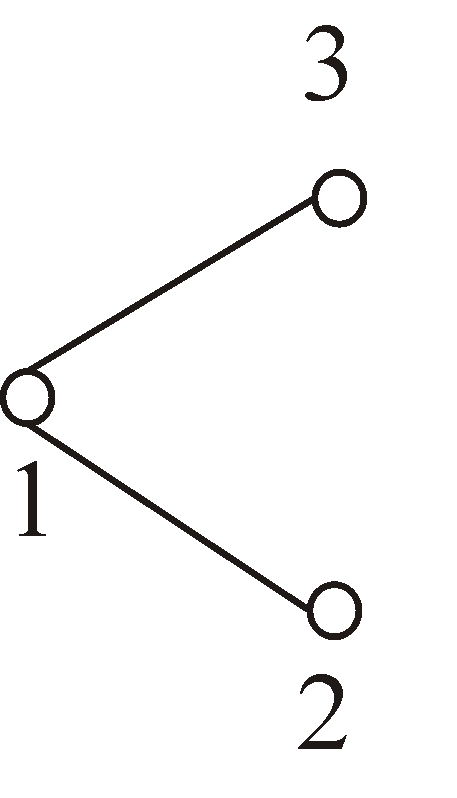}
\end{figure*}

\subsubsection{Some useful lemmas}

\noindent Our main goal is to find those highest weights $\lambda\in
H^*$ for which $Z(\lambda)$ is an irreducible highest weight module
of the simple $n$-Lie algebra $A$.\\
\\
 \noindent We start by showing that if $Q(A)\subseteq J(\lambda)$,
then the universal enveloping algebra $U$ of the n-Lie algebra $A$,
acts on $U/J(\lambda)=Z(\lambda)$, which means that $Z(\lambda)$ is
an irreducible $n$-Lie module. In the following Lemma we will prove
this claim and its converse\footnote{We have already mentioned this
equivalence above.}.
\begin{lema}
$Q(A)\subseteq J(\lambda)\textrm{ if and only if } U(V)\textrm{ acts
on }U/J(\lambda)$.
\end{lema}
\begin{proof} \noindent Let $\rho:U\rightarrow
\textrm{End}(U/J(\lambda))$ be the representation, given by left
multiplication
\[\rho(u)(x)=u.x.\] Then, obviously, $\textrm{Ker}\rho\subseteq
J(\lambda)$.\\
\noindent "$\Leftarrow$"\\
\noindent Assume that $U/Q(A)$ acts on $Z(\lambda)$. Then, $Q(A)\subseteq\textrm{Ker}\rho\subseteq J(\lambda)$.\\
\noindent "$\Rightarrow$"\\
\noindent Assume that $Q(A)\subseteq J(\lambda)$. We want to show
that $Q(A)$ is also a subset of $\textrm{Ker}\rho$.\\
Let $q\in Q(A)$ and $[u]\in U/J(\lambda)$, where $u\in U$. Then
\begin{align*}
\rho(q)([u])=\rho(q\cdot u)(1)\in \rho(Q(V))(1)\subseteq
\rho(J(\lambda))(1)=J(\lambda).
\end{align*}
Thus $\rho(q)([u])=0$ and we are done.
\end{proof}
\noindent We can rephrase this Lemma as follows: $Z(\lambda)$ is a
representation of $A$ if and only if $Q(A)\subseteq J(\lambda)$.
Thus, we want to see for which $\lambda\in H^*$ the inclusion above
holds. The next Lemma gives us a useful method for checking this
inclusion.
\begin{lema}
$Q(A)\nsubseteq J(\lambda)\textrm{ if and only if }
Q(A)+I(\lambda)=U.$
\end{lema}
\begin{proof}
\noindent "$\Leftarrow$"\\
\noindent Assume that $Q(A)+I(\lambda)=U$. Let us suppose that
$Q(A)\subseteq J(\lambda)$. It follows that
$Q(A)+I(\lambda)\subseteq J(\lambda)$. But $J(\lambda)$ is strictly
included in $U$. Contradiction.\\
\noindent "$\Rightarrow$"\\
\noindent Next we prove the converse implication. Assume that
$Q(A)\nsubseteq J(\lambda)$ and suppose that $Q(A)+I(\lambda)\neq
U$. Let $M$ be the maximal left ideal of $U$ such that
$Q(A)+I(\lambda)\subseteq M$. By definition $J(\lambda)$ is the
unique maximal left ideal of $U$ which contains $I(\lambda)$, hence
$M=J(\lambda)$. It follows that $Q(A)+I(\lambda)\subseteq
J(\lambda)\Rightarrow Q(A)\subseteq J(\lambda)$. Contradiction.
\end{proof}
\noindent $Q(A)+I(\lambda)=U$ is equivalent to the fact that $1\in
Q(A)+I(\lambda)$, where $1$ is the unit in $U$. Denote by
$\hat{Q}(A):=\frac{Q(A)+I(\lambda)}{I(\lambda)}\subseteq
V(\lambda)=U/I(\lambda)$ and by $\mathbb{1}:=1+I(\lambda)$. We can
rewrite the above equality as $\mathbb{1}\in \hat{Q}(A)$.\\
\noindent We define the following map:\[pr_{\lambda}:V(\lambda)\to
\textrm{Span}\mathbb{1}=V(\lambda)_{\lambda},\] where
$V(\lambda)_{\lambda}$ denotes the 1-dimensional subspace of
$V(\lambda)$ with weight $\lambda$. With this notation we obtain the
following equivalence:\[\mathbb{1}\in\hat{Q}(A)\textrm{ if and only
if there exists } x\in \hat{Q}(A)\textrm{ s.t. }pr_{\lambda}(x)\neq
0.\] The implication "$\Rightarrow$" is trivial, hence in order to
convince ourself that the above equivalence holds we only need to
check the converse implication.
\[\textrm{Let }x\in \hat{Q}(A) \textrm{ with } pr_{\lambda}(x)\neq 0,\textrm{ then }\hat{Q}(A)\nsubseteq \textrm{Ker} (pr_{\lambda}).\]
\[\textrm{On the other hand }\hat{J}(\lambda)=J(\lambda)/I(\lambda)\subseteq\textrm{Ker} (pr_{\lambda})\]
\[\Rightarrow \hat{Q}(A)\nsubseteq \hat{J}(\lambda)\Rightarrow \mathbb{1}\in \hat{Q}(A).\]

\noindent The next lemma is the most important one in this section.

\begin{lema}
$A$ has a representation on $Z(\lambda)$ if and only if
$pr_\lambda(\hat{R})=0$, where
$\hat{R}=\frac{R+I(\lambda)}{I(\lambda)}$.
\end{lema}
\begin{proof}
\noindent By the last equivalence above and the lemmas before, it
follows that
\[A\textrm{ has a representation on
}Z(\lambda)\textrm{ if and only if }\hat{Q}(A)\subseteq
\textrm{Ker}(pr_\lambda).\] Observe that the last equivalence holds
in general. Hence, we just have to show that
\[\hat{Q}(A)\subseteq
\textrm{Ker}(pr_\lambda)\textrm{ if and only if
}\hat{R}\subseteq\textrm{Ker}(pr_\lambda).\] Since
$\hat{R}\subseteq\hat{Q}(A)$, the implication "$\Rightarrow$"
is obvious.\\
\noindent "$\Leftarrow$"\\
\noindent Assume that $\hat{R}\subseteq \textrm{Ker}(pr_\lambda)$
and let $[q]\in \hat{Q}(A)$, where $q\in Q(A)$. Since
$Q(A)=URU=UR=RU$, the ideal $Q(A)$ is spanned by elements of the
form $q=r\cdot u$, where $r\in R$ and $u\in U$, hence $[q]=[r\cdot
u]=r\cdot[u]$, where $[u]\in U/I(\lambda)=V(\lambda)$. This implies
that $[u]=\sum[x_{\alpha_1}x_{\alpha_2}\ldots x_{\alpha_k}]$, where
$\alpha_i\in\Phi^-$, for all $i$ in the set $\{1,\ldots,k\}$. This
shows that $\hat{Q}(A)$ is spanned by elements of the form $r\cdot
x_{\alpha_1}x_{\alpha_2}\ldots x_{\alpha_k}\cdot \mathbb{1}$, where
$r\in R$ and $\alpha_i\in\Phi^-$. Since $R$ is a finite dimensional
$\wedge^{n-1}A\cong so(n+1)$-module and $so(n+1)$ is (semi-)simple,
it follows that $R$ is decomposable into weight spaces. Let
$\{r_1,\ldots,r_k\}$ be a basis of $R$, where $weight(r_i)=\mu_i$.
Then, $\hat{Q}(A)$ is spanned by elements of the form $r_i\cdot
x_{\alpha_1}x_{\alpha_2}\ldots x_{\alpha_k}\cdot \mathbb{1}$. To see
if $\hat{Q}(A)$ is indeed contained in $\textrm{Ker}(pr_\lambda)$,
we have to compute $pr_\lambda(r_i\cdot
x_{\alpha_1}x_{\alpha_2}\ldots x_{\alpha_k}\cdot \mathbb{1})$. Since
$weight(\mathbb{1})=\lambda$, it follows that we need to compute
this projection for those elements for which
$weight(x_{\alpha_1}x_{\alpha_2}\ldots x_{\alpha_k})+weight(r_i)=0$,
i.e.
$\alpha_1+\alpha_2+\ldots+\alpha_k+\mu_i=0$:\[pr_\lambda(r_i\cdot
x_{\alpha_1}x_{\alpha_2}\ldots x_{\alpha_k}\cdot
\mathbb{1})=pr_\lambda([ r_i, x_{\alpha_1}] x_{\alpha_2}\ldots
x_{\alpha_k}\cdot \mathbb{1})+pr_\lambda( x_{\alpha_1}\cdot r_i\cdot
x_{\alpha_2}\ldots x_{\alpha_k}\cdot \mathbb{1}).\] We already know
that in the first term of the sum, $[ r_i, x_{\alpha_1}]$ is a
linear combination of generators of $Q(A)$ ($R$ being a module for
$so(n+1)$), hence $[ r_i, x_{\alpha_1}]=\sum_j a_jr_j$. The second
term of the sum is zero, since $weight(r_i\cdot x_{\alpha_2}\ldots
x_{\alpha_k}\cdot \mathbb{1}))\succ \lambda$. Thus, by an inductive
argument, we obtain \[pr_\lambda(r_i\cdot
x_{\alpha_1}x_{\alpha_2}\ldots x_{\alpha_k}\cdot \mathbb{1})=\sum_j
a_j pr_\lambda(r_j\cdot\mathbb{1})=0.\] This proves the claim of the
Lemma.
\end{proof}
\noindent Let $r_i\in R$. If $weight(r_i)=\mu_i\neq 0$ then
$weight(r_i\cdot\mathbb{1})\neq\lambda $ and it follows that
$pr_\lambda(r_i\cdot\mathbb{1})=0$. This shows that in order for
$Z(\lambda)$ to be an $n$-Lie module, we just have to look at those
elements $r_0\in R$ for which $weight(r_0)=0$ and insure that
$pr_\lambda(r_0\cdot\mathbb{1})=0$.

\section{Main theorems and their proofs}

\noindent In this section we will concentrate on the important
results of this paper. The basic Lie algebra of the simple $n$-Lie
algebra $A$ is $so(n+1)$ and we denote by $H$ a Cartan subalgebra of
this Lie algebra. Recall that our goal is to determine for which
$\lambda\in H^*$, the two-sided ideal $Q(A)$ acts trivially on the
irreducible, highest weight module $Z(\lambda)$. This will insure
that $Z(\lambda)$ is an $n$-Lie module of $A$. We will first explain
the notation used in the statements of the theorems.

\subsection{so(n+1)}

\noindent Until now, the $(n+1)$-dimensional, complex vector space
$A$ came attached with the standard basis $e_1,\ldots,e_{n+1}$, the
inner product $\langle e_i,e_j\rangle=\delta_{i,j}$ and the orientation form
$e_1\wedge\ldots\wedge e_{n+1}$. This was useful for us because of
several reasons: the easy expression of the $n$-ary bracket, the
simple form of the generators of $Q(A)$ being just two examples we
have encountered so far. It will be useful for us, further on, to have a complex bilinear form instead of an inner product on A. Denote by $(\cdot,\cdot)$ the bilinear form on A
defined by
\begin{equation}
\label{j3}
(e_i,e_j)=\delta_{i,j}, \quad 1\le i,j\le n+1.
\end{equation}
Since $\textrm{End}(A)$ and $A\otimes A$ are isomorphic,
$so(n+1)\subset\textrm{End}(A)$ can be identified with $\wedge^2A$
(cf (\ref{j2})). To give the root basis of this $so(n+1)$ it will be
convenient to have an other basis of $A$. For this we need to take
into account the parity of $n+1$. If $n+1=2N$, we define
\begin{equation}
\label{j4} v_{\pm j}=\frac{e_{2j-1}\mp i e_{2j}}{\sqrt 2},\quad
j=1,2,\ldots, N,
\end{equation}
then a basis of $A$ will be
denoted by $v_{-N},\ldots,v_{-1},v_1,\ldots,v_N$, and
\[(v_i,v_j)=\delta_{i+j,0}.\]
If, on the other hand, $n+1=2N+1$ then
we need one more element, viz. $v_{0}:=e_{2N+1}$.
Then the basis of $A$ is given by
$v_{-N},\ldots,v_{-1},v_0,v_1,\ldots,v_N$, and again
$(v_i,v_j)=\delta_{i+j,0}$. Hence, a basis of $so(n+1)\simeq \wedge^2A$ is given
by elements of the form (cf. (\ref{j2}))
\[
v_{j}\wedge v_{k}=\frac{v_j\otimes v_k-v_k\otimes v_j}{2},\
\mbox{where } -N\le j<k\le N
\] and
$j,k\ne 0$ if $n+1$ is even.
The relation between the two notations used is given by:
\[v_{\nu j}\wedge v_{\mu k}=\frac{1}{2}(e^{2j-1,2k-1}-\nu\mu e^{2j,2k}-i(\nu e^{2j,2k-1}+\mu e^{2j-1,2k})),\]
\[v_0\wedge v_{\nu j}=\frac{1}{\sqrt{2}}(e^{2j-1,2N+1}-i\nu e^{2j,2N+1}).\]
The natural bilinear form on $so(n+1)$ is the normalized trace form
$(M_1|M_2):=\frac{1}{2}\mbox{trace}(M_1M_2)$, which is
nondegenerate,invariant and  symmetric. This form is induced by a
natural form on $\textrm{End}(A)\simeq A\otimes A$:
\[
((u_1\otimes u_2)|(w_1\otimes w_2))=2(u_1,w_2)(u_2,w_1).
\]
\noindent Define
\begin{equation}
\label{j5} \epsilon_j:=ie^{2j-1,2j}=i(e_{2j-1}\wedge
e_{2j})=i\frac{e_{2j-1}\otimes e_{2j}-e_{2j}\otimes
e_{2j-1}}{2},\quad\mbox{where } 1\leq j\leq N.
\end{equation}
Then $\epsilon_j=v_j\wedge v_{-j}$ and
$H:=\oplus_{i=1}^{N}\mathbb{C}\epsilon_i$ is a Cartan subalgebra of
the Lie algebra $so(n+1)$. Let $h\in H$, then the commutator of $h$
and $v_{\nu j}\wedge v_{\mu k}$ is: \[[h, v_{\nu j}\wedge v_{\mu
k}]=(\nu\epsilon_j+\mu\epsilon_k|h)v_{\nu j}\wedge v_{\mu k}.\]
\[[ h,v_0\wedge v_{\nu j}]=(\nu\epsilon_j|h)v_0\wedge v_{\nu j}.\]
\begin{remark}
$v_0\wedge v_{\nu j}$ should be considered to be $ v_{\nu j}\wedge
v_0$ in case $\nu=-1$, and not $-v_{\nu j}\wedge v_0$ as would
result from the antisymmetry of the exterior product. For this
reason we will write $\nu (v_0\wedge v_{\nu j})$.
\end{remark}
\noindent Thus, if $n+1=2N$, we define the set of roots
\[\Phi=\{\pm(\epsilon_i\pm\epsilon_j)|1\leq i\neq j\leq N\}\] and a
base for this root system is given by
\[\Delta=\{\epsilon_1-\epsilon_2,\epsilon_2-\epsilon_3,\ldots,\epsilon_{N-1}-\epsilon_N,\epsilon_{N-1}+\epsilon_N\}.\]
Hence, we obtain the following root space decomposition for
$so(2N)$:\[so(2N)=\bigoplus_{\substack{1\leq j<k\leq N\\
\mu\in\{+1,-1\}}}\mathbb{C}(v_{ -j}\wedge v_{\mu
k})\bigoplus\bigoplus_{1\leq j\leq
N}\mathbb{C}\epsilon_j\bigoplus\bigoplus_{\substack{1\leq j<k\leq N\\
\mu\in\{+1,-1\}}}\mathbb{C}(v_{j}\wedge v_{\mu k}).\]
%\noindent Our short-term goal is to find a root space decomposition
%of the Lie algebra $so(2N+1)$, such that we can apply the same
%strategy as in the previous case. For this purpose we define the set
%of roots to be
If $n+1=2N+1$, then the set of roots is
\[\Phi=\{\pm(\epsilon_i\pm\epsilon_j)|1\leq i\neq j\leq N\}\cup\{\pm\epsilon_i|1\leq i\leq
N\},\] while a base for this root system is given by
\[\Delta=\{\epsilon_1-\epsilon_2,\epsilon_2-\epsilon_3,\ldots,\epsilon_{N-1}-\epsilon_N,\epsilon_N\}.\]
This allows us to give the following root space decomposition of the
Lie algebra $so(2N+1)$:
\[so(2N+1)=\bigoplus_{\substack{1\leq j<k\leq N\\ \mu\in \{+1,-1\}}}\mathbb{C}(v_{-j}\wedge v_{\mu k})\bigoplus_{1\leq j\leq N}\mathbb{C}(v_{-j}\wedge v_0)\bigoplus_{1\leq j\leq N}\mathbb{C}\epsilon_j\bigoplus\]\[\bigoplus_{1\leq j\leq
N}\mathbb{C}(v_{0}\wedge v_j)\bigoplus_{\substack{1\leq j<k\leq N\\
\mu\in \{+1,-1\}}}\mathbb{C}(v_{j}\wedge v_{\mu k}).\] The first
line of the formula above contains the negative side of the root
space decomposition and the Cartan subalgebra, while on the second
line just the positive side is listed.

\subsection{Statements of the main theorems}

Recall that $H:=\oplus_{i=1}^{N}\mathbb{C}\epsilon_i$ is a Cartan
subalgebra of the Lie algebra $so(n+1)$. Let $\lambda$ be the
highest weight of the highest weight module $Z(\lambda)$ and denote
by $\lambda_i$ the value of $\lambda$ on $\epsilon_i$.

\begin{thm} \label{t4.1}Let $n\geq 3$. The highest weight, irreducible
representation $Z(\lambda)$ of $so(n+1)$ is a highest weight,
irreducible representations of the simple $n$-Lie algebra $A$ if and
only if $\lambda\in H^*$ is such that
$\lambda_1=\lambda_2=\ldots=\lambda_{t-1}=-1$, $\lambda_t=x\in
\mathbb{C}$ and
$\lambda_{t+1}=\ldots=\lambda_{\lfloor\frac{n+1}{2}\rfloor}=0$, for
some $1\leq t\leq \lfloor\frac{n+1}{2}\rfloor$.
\end{thm}
\noindent From general theory of irreducible Lie algebra
representations it follows that if $x\in\mathbb{Z}_+$ and $t=1$ we
obtain a finite dimensional, irreducible representation of $A$ and
otherwise
an infinite dimensional, irreducible one. Thus, we have recovered the result in \cite{Dzhumadil'daev}. Our proof however, will be different from the one presented there.\\
\\
\noindent When stating the above theorem in terms of the fundamental
weights of $so(n+1)$, one needs to be careful about the distinction
between the two cases: $n+1$ is even or $n+1$ is odd. This gives
Theorems \ref{j6} and \ref{j7}.
%We will state
%these two theorems below.
%\begin{thm}
%Let $n\geq 3$, $n+1=2N$ and $t\in\{1,\ldots,N\}$. Denote by
%$\pi_1,\ldots,\pi_N$ the fundamental weights of $so(2N)$. Then,
%$Z(\lambda)$ is an irreducible representation of $A$ if and only if
%\begin{displaymath}
%\lambda=\left\{
%\begin{array}{ll}
%x\pi_1 & t=1,\\
%(-1-x)\pi_{t-1}+x\pi_t & 1<t<N-1,\\
%(-1-x)\pi_{N-2}+x\pi_{N-1}+x\pi_N & t=N-1,\\
%(-1-x)\pi_{N-1}+(-1+x)\pi_N & t=N,
%\end{array}
%\right.
%\end{displaymath}
% where $x\in\mathbb{C}.$
%\end{thm}
%\begin{thm}
%Let $n\geq 3$, $n+1=2N+1$ and $t\in\{1,\ldots,N\}$. Denote the
%fundamental weights of $so(2N+1)$ by $\pi_1,\ldots,\pi_N$. Then,
%$Z(\lambda)$ is an irreducible highest weight representation of $A$
%if and only if
%\begin{displaymath}
%\lambda=\left\{
%\begin{array}{ll}
%x\pi_1 & t=1,\\
%(-1-x)\pi_{t-1}+x\pi_t & 1<t\leq N,
%\end{array}
%\right.
%\end{displaymath}
%where $x\in\mathbb{C}.$
%\end{thm}

\noindent The rest of this section will contain the proof of Theorem
\ref{t4.1}. We will follow the strategy described above, namely we
will determine the generators of $Q(A)$ of weight zero and impose
the condition that they must act trivially on $Z(\lambda)$. These
conditions can be expressed as the zero set of a set of polynomials
in $\lambda_1,\ldots,\lambda_{\lfloor\frac{n+1}{2}\rfloor}$. These
solutions can be read in the statement of the theorem. There will be
two proofs given: a computational proof which will differentiate
between the two cases: $n+1$ is even or odd; and a second proof
which will make use of the graphical language introduced before.
This second proof has the advantage of working in both cases while
no additional distinction needs to be made. This graphical proof
also has a disadvantage: as required, it does supply the polynomials
we where talking about, but it does not supply any indication why
the obtained polynomials suffice.

\subsection{n+1 is even}
\noindent First, we treat the case: $n+1=2N$. This means that the Lie algebra $\wedge^{n-1}A$ is $so(2N)$.\\
\\
\noindent We want to compute the generators of $Q(A)$ in terms of
the elements $v_{\nu j}\wedge v_{\mu k}$. By inverting the matrix
which defines the $v_{\nu j}\wedge v_{\mu k}$'s in terms of the
$e^{j,k}$'s we obtain the following equality.
\[\left(\begin{array}{c}
e^{2j-1,2k-1}\\
e^{2j,2k}\\
e^{2j,2k-1}\\
e^{2j-1,2k}
\end{array}
\right)=\frac{1}{2}\left(
\begin{array}{cccc}
1&1&1&1\\
-1&1&1&-1\\
i&-i&i&-i\\
i&i&-i&-i
\end{array}
\right) \left(
\begin{array}{c}
v_{j}\wedge v_{k}\\
v_{-j}\wedge v_{k}\\
v_{j}\wedge v_{-k}\\
v_{-j}\wedge v_{-k}
\end{array}
\right)\] \noindent We will avoid to write long, tedious
computations, and jump ahead to the final result. The most important
piece of information is the matrix above.\\
\\
\noindent Recall that
$x_{i_1,i_2,i_3,i_4}=e^{i_1,i_2}e^{i_3,i_4}-e^{i_1,i_3}e^{i_2,i_4}+e^{i_1,i_4}e^{i_2,i_3}$.
This formula can also be rewritten
as:\[x_{i_1,i_2,i_3,i_4}=\frac{1}{8}\sum_{\sigma\in
S_4}sgn(\sigma)e^{i_{\sigma(1)},i_{\sigma(2)}}e^{i_{\sigma(3)},i_{\sigma(4)}}.\]
It follows that for any $\tau$ in $S_4$,
\[x_{i_{\tau(1)},i_{\tau(2)},i_{\tau(3)},i_{\tau(4)}}=sgn(\tau)x_{i_1,i_2,i_3,i_4}.\]
Therefore, if two indexes are equal, the formula above tells us that
$x_{i_1,i_2,i_3,i_4}=0$. Hence
\begin{align*}&\textrm{Span}\{x_{i_1,i_2,i_3,i_4}\}_{1\leq i_1,i_2,i_3,i_4\leq
2N}=\\&=\textrm{Span}\{\frac{1}{8}\sum_{\sigma\in
S_4}sgn(\sigma)e^{i_{\sigma(1)},i_{\sigma(2)}}e^{i_{\sigma(3)},i_{\sigma(4)}}\}_{1\leq
i_1,i_2,i_3,i_4\leq 2N}=\\&=R.
\end{align*}
\\
\noindent Now we are ready to give the generators of $Q(A)$
expressed in the new notation:
\[\textrm{Span}\{x_{i_1,i_2,i_3,i_4}\}_{1\leq i_1,i_2,i_3,i_4\leq
2N}=\] \begin{align*}=\textrm{Span}\{\frac{1}{2} ((v_{i}\wedge
v_{j})(v_{k}\wedge v_{l})-(v_{i}\wedge v_{k})(v_{j}\wedge
v_{l})+(v_{i}\wedge v_{l})(v_{j}\wedge v_{k})+\\ +(v_{k}\wedge
v_{l})(v_{i}\wedge v_{j})-(v_{j}\wedge v_{l})(v_{i}\wedge
v_{k})+(v_{j}\wedge v_{k})(v_{i}\wedge v_{l}))\},\end{align*} where
$i,j,k,l$ are considered to be ordered and they range from $-N$ to
$N$, excluding 0. A generator, expressed in this notation, will be
denoted by $v_{i,j,k,l}(\alpha,\beta,\gamma,\delta)$, where
$\alpha,\beta,\gamma$ and $\delta$ represent the signs of the
indices. These indices range now between 1 and $N$.\\
\\
\noindent Recall that $x_{a,b,c,d}$ is zero as soon as any two
indices are equal. Without loss of generality we may assume that
$a\leq b\leq c\leq d$. For $
v_{a,b,c,d}(\alpha,\beta,\gamma,\delta)$ this fact does not hold
anymore, i.e. if any two indices are equal, then $
v_{a,b,c,d}(\alpha,\beta,\gamma,\delta)$ is
zero if the corresponding signs are also equal.\\
\\
\noindent It is easy to see that the 6 terms in the expression of
the generator $ v_{a,b,c,d}(\alpha,\beta,\gamma,\delta)$ all have
the same weight, namely
\[weight(v_{a,b,c,d}(\alpha,\beta,\gamma,\delta))=\alpha\epsilon_a+\beta\epsilon_b+\gamma\epsilon_c+\delta\epsilon_d.\]
Since we are looking for those elements $r_0\in R$ for which
$weight(r_0)=0$, it follows that for
$v_{a,b,c,d}(\alpha,\beta,\gamma,\delta)$ to be such an element we
must have \[a=b,c=d\textrm{ and
}\alpha=\gamma=1,\beta=\delta=-1.\footnote{Recall that $a\leq b\leq
c\leq d$ and that the equality of 3 consecutive indexes is not
possible. Observe that for different choices of signs, the resulting
generator gets multiplied by $\pm 1$.}\] \noindent In this case, our
generator of $Q(A)$ becomes
\begin{align*}
&v_{a,a,c,c}(1,-1,1,-1)=\\ &\frac{1}{2}((v_{a}\wedge
v_{-a})(v_{c}\wedge v_{-c})-(v_{a}\wedge v_{c})(v_{-a}\wedge
v_{-c})+(v_{a}\wedge v_{-c})(v_{-a}\wedge v_{c})+\\&+(v_{c}\wedge
v_{-c})(v_{a}\wedge v_{-a})-(v_{-a}\wedge v_{-c})(v_{a}\wedge
v_{c})+(v_{-a}\wedge v_{c})(v_{a}\wedge v_{-c})).\end{align*}

\noindent By using the fact that $h\mathbb{1}=\lambda(h)\mathbb{1}$
for all $h\in H$, while $v_j\wedge v_{\nu k}\mathbb{1}=0$ for any
$1\leq j<k\leq 2N$ and $\nu\in\{+1,-1\}$, we can compute
$pr_\lambda(v_{a,a,c,c}(1,-1,1,-1)\cdot\mathbb{1})$.

\begin{align*}
&v_{a,a,c,c}(1,-1,1,-1)\cdot\mathbb{1}=\\
&\frac{1}{2}(2\epsilon_a\epsilon_c+(-\epsilon_a+\epsilon_c)-(-\epsilon_a-\epsilon_c))\cdot\mathbb{1}=\\
&(\epsilon_a\epsilon_c+\epsilon_c)\cdot\mathbb{1}
\end{align*}
\noindent Thus,
\[pr_\lambda(v_{a,a,c,c}(1,-1,1,-1)\cdot\mathbb{1})=\lambda(\epsilon_c)(\lambda(\epsilon_a)+1).\]

\noindent Above we have established that $Z(\lambda)$ is an
irreducible $A$-module if and only if $\hat{R}\subseteq\textrm{Ker}(
pr_\lambda)$. We have already seen that the only elements $r\in R$
which might cause a problem are those of weight equal to zero. For
the case "$n+1$ is even" the only such element is
$v_{a,a,c,c}(1,-1,1,-1)$, whose projection is equal to
$\lambda(\epsilon_c)(\lambda(\epsilon_a)+1)$. Hence, if $\lambda\in
H^*$ satisfies the condition that for any $a<c$,
$\lambda(\epsilon_c)(\lambda(\epsilon_a)+1)=0$, then $Z(\lambda)$ is
indeed an irreducible $A$-module.

\subsection{n+1 is odd}
\noindent Next, we treat the case: $n+1=2N+1$. Our basic Lie algebra $\wedge^{n-1}A$ is now $so(2N+1)$ with basis
%\\
%\noindent As above, a basis of $so(2N+1)$ is given by the elements
%$e^{ij}:=E_{ij}-E_{ji}$, where $1\leq i<j\leq 2N+1$. Define
%$\epsilon_j:=ie^{2j-1,2j}$, $1\leq j\leq N$, then
%$H=\oplus_{i=1}^n\mathbb{C}\epsilon_i$ is a Cartan subalgebra of the
%Lie algebra $so(2N+1)$. The normalized trace form (introduced in the
%"$n+1$ even" case) is a nondegenerate, symmetric, bilinear form,
%whose restriction to $H$ is nondegenerate:
%$(\epsilon_i|\epsilon_j)=\delta_{ij}$.\\
%\\
%\noindent On the other hand, we introduced a new notation for this
%basis. Denote by $v_{-N},\ldots,v_0,\ldots,v_N$ the basis of the
%simple Lie algebra $A$. Then, a basis for the basic Lie algebra
%$so(2N+1)$ is
given by elements of the form $v_i\wedge v_j$, where
$i$ and $j$ are included in the set $\{-N,\ldots,0,\ldots,N\}$.
We will apply the same strategy as in the even case,
namely: rewrite the generators of the ideal $Q(A)$ in terms of the
elements above, and afterwards find those which have weight
zero.\\
\\
\noindent Observe first that the elements $v_{\nu j}\wedge v_{\mu
k}$ have the same definition as in the previous case. This proves
that: \begin{align*} &\textrm{Span}\{x_{i_1,i_2,i_3,i_4}\}_{1\leq
i_1,i_2,i_3,i_4\leq
2N}=\\
&\textrm{Span}\{(v_{\alpha a}\wedge v_{\beta b})(v_{\gamma c}\wedge
v_{\delta d})+(v_{\alpha a}\wedge v_{\delta d})(v_{\beta b}\wedge
v_{\gamma c})- (v_{\alpha a}\wedge v_{\gamma c})(v_{\beta b}\wedge
v_{\delta
d})+\\
&+(v_{\gamma c}\wedge v_{\delta d})(v_{\alpha a}\wedge v_{\beta
b})+(v_{\beta b}\wedge v_{\gamma c})(v_{\alpha a}\wedge v_{\delta
d})-(v_{\beta b}\wedge v_{\delta d})(v_{\alpha a}\wedge v_{\gamma
c})\},\end{align*} where $1\leq a,b,c,d\leq N$ and $
\alpha,\beta,\gamma,\delta\in \{\pm 1\}$. This shows that if all
indices involved in $x_{a,b,c,d}$, are strictly less than $2N+1$,
then the "weight-zero" relations we are looking for, are those
obtained in the "$n+1$-even" case. Hence, we only need to look for
those "weight-zero" relations,
$x_{a,b,c,d}$, which involve the indices $a\leq b\leq c <d=2N+1$.\\
\\
\noindent By inverting the matrix which defines the elements
$v_0\wedge v_a$ and $v_{-a}\wedge v_0$ we obtain that:
\[e^{2a-1+\alpha,2N+1}=\frac{1}{\sqrt{2}}\sum_{\nu\in\{-1,1\}}\nu(i\nu)^\alpha v_0\wedge v_{\nu a},\]
where $\alpha\in\{0,1\}$. This gives us the following equality:
\begin{align*}
&x_{2a-1+\alpha,2b-1+\beta,2c-1+\gamma,2N+1}=\\
&\frac{1}{4\sqrt{2}}\sum_{\nu,\mu,\omega\in\{1,-1\}}(i\nu)^{\alpha}(i\mu)^{\beta}(i\omega)^{\gamma}\cdot\\
&\bigg(\omega (v_{\nu a}\wedge v_{\mu b})(v_0\wedge v_{\omega c})+
\nu (v_0\wedge v_{\nu a})(v_{\mu b}\wedge v_{\omega c})-\mu (v_{\nu a}\wedge v_{\omega c})(v_0\wedge v_{\mu b})\\
&+\omega (v_0\wedge v_{\omega c})(v_{\nu a}\wedge v_{\mu b})+\nu
(v_{\mu b}\wedge v_{\omega c})(v_0\wedge v_{\nu a})-\mu (v_0\wedge
v_{\mu b})(v_{\nu a}\wedge v_{\omega c})\bigg).
\end{align*}
At this point we could prove that the span of the elements on the
left-hand-side is equal to the span of the elements on the
right-hand-side and follow the same guide line as before. But, there
is no need for this, since the span of the elements on the left is
obviously included in the span of the elements on the right;
moreover the elements on the right have weight:
\[weight=\nu\epsilon_a+\mu\epsilon_b+\omega\epsilon_c.\]
For this weight to be zero, we would need that $a=b=c$. However, if
this is the case, then $\nu+\mu+\omega$ can not be 0, since
$\nu,\mu,\omega\in\{1,-1\}$. Thus, no new polynomials are added to
the set previously obtained.

\subsection{Making use of the graphical interpretation}

\noindent In the following we will demonstrate another way of
obtaining the same polynomials.

\noindent By the definition of the ideal $I(\lambda)$, it follows
    that in $V(\lambda)$ the elements $v_j\wedge v_k$ and
    $v_j\wedge v_{-k}$ are zero, where $j<k$. We use the graphical interpretation
    of  these elements to obtain the following equalities:\\
    \\
    \noindent
    $v_j\wedge v_k=\frac{1}{2}(\includegraphics[width=0.1\textwidth]{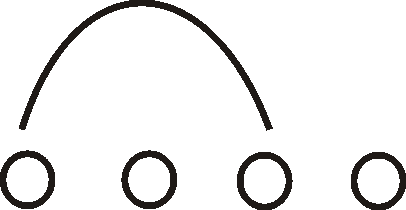}-\includegraphics[width=0.1\textwidth]{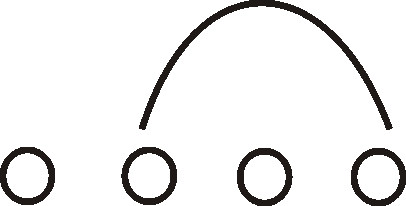}-i\includegraphics[width=0.1\textwidth]{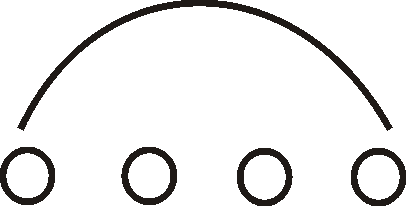}-i\includegraphics[width=0.1\textwidth]{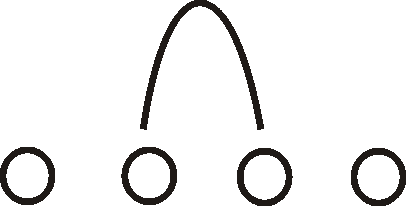})=0$,\\
    $v_j\wedge v_{-k}=\frac{1}{2}(\includegraphics[width=0.1\textwidth]{gplus1}+\includegraphics[width=0.1\textwidth]{gplus2}+i\includegraphics[width=0.1\textwidth]{glus3}-i\includegraphics[width=0.1\textwidth]{gplus4})=0$.\\
\\
    \noindent Remember that we chose $j<k$ and observe that in these equalities the four points are numbered by
    $2j-1,2j,2k-1,2k$. The computations presented below, should be seen as taking place in $U/(Q(A)+Ann(V(\lambda)))$. By adding and subtracting the two
    equalities we obtain the following two relations:\\
\\
    \noindent
    $\includegraphics[width=0.1\textwidth]{gplus1}=i\includegraphics[width=0.1\textwidth]{gplus4}$,\\
    $\includegraphics[width=0.1\textwidth]{gplus2}=-i\includegraphics[width=0.1\textwidth]{glus3}$.\\
\\
    \noindent They tell us that if an arc connects two points in
    the diagram which have different parities, then we can shift the
    left leg of the arc, such that the result is an
    arc between two points of the same parity. By doing this we
    acquire either $+i$ or $-i$ in front of the diagram.\\
\\
\noindent The ideal $I(\lambda)$ is a left ideal, hence:\\
\\
\noindent
$(\includegraphics[width=0.1\textwidth]{gplus1}+i\includegraphics[width=0.1\textwidth]{gplus4})(\includegraphics[width=0.1\textwidth]{gplus2}+i\includegraphics[width=0.1\textwidth]{glus3})=0$\\
\noindent and\\
$(\includegraphics[width=0.1\textwidth]{gplus2}-i\includegraphics[width=0.1\textwidth]{glus3})(\includegraphics[width=0.1\textwidth]{gplus1}-i\includegraphics[width=0.1\textwidth]{gplus4})=0.$\\
\\
\noindent Recall that the Cartan subalgebra $H$ was defined as
$\bigoplus_{i=1}^N\mathbb{C}\epsilon_i$, where
$\epsilon_j=ie^{2j-1,2j}=i\includegraphics[width=0.04\textwidth]{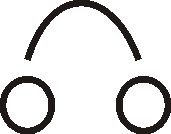}$
(a simple diagram with one arc connecting the points $2j-1$ and
$2j$). The element $\epsilon_j$ acts on the highest weight module by
multiplication with $\lambda(\epsilon_j)$, which we will denote by
$\lambda_j$. Using this fact we compute the two equalities
above.\\
\noindent The first equality becomes:
\begin{align*}
&(\includegraphics[width=0.1\textwidth]{gplus1}+i\includegraphics[width=0.1\textwidth]{gplus4})(\includegraphics[width=0.1\textwidth]{gplus2}+i\includegraphics[width=0.1\textwidth]{glus3})=\\
&=\includegraphics[width=0.1\textwidth]{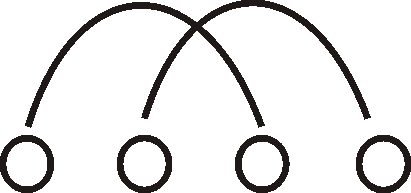}-\includegraphics[width=0.1\textwidth]{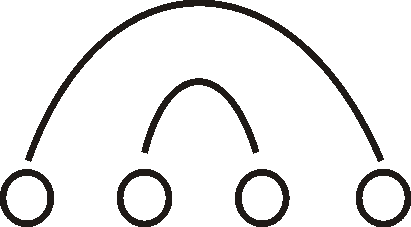}+i(\includegraphics[width=0.1\textwidth]{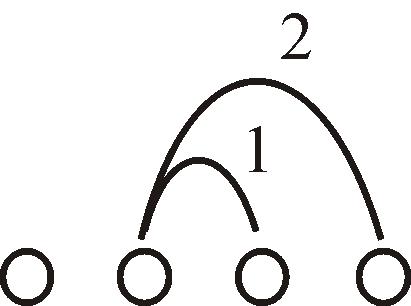}+\includegraphics[width=0.1\textwidth]{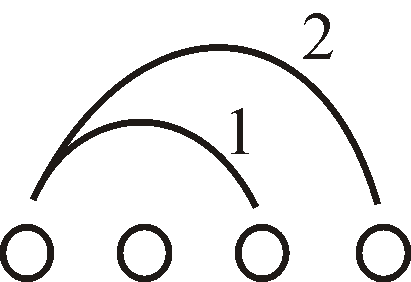})=\\
&=\includegraphics[width=0.1\textwidth]{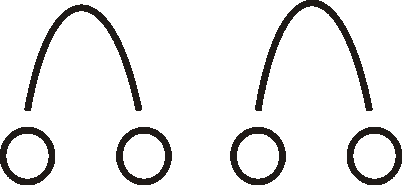}+i(\includegraphics[width=0.1\textwidth]{ating12}+\includegraphics[width=0.1\textwidth]{ating21})=\\
&=(i\includegraphics[width=0.1\textwidth]{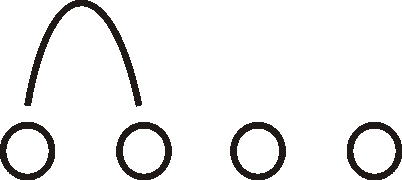})(-i\includegraphics[width=0.1\textwidth]{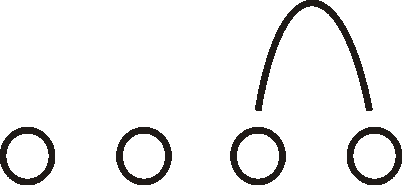})+i(\includegraphics[width=0.1\textwidth]{ating12}+\includegraphics[width=0.1\textwidth]{ating21})=\\
&=-\lambda_j\lambda_k+i(\includegraphics[width=0.1\textwidth]{ating12}+\includegraphics[width=0.1\textwidth]{ating21})=\\
&=-\lambda_j\lambda_k+i(\includegraphics[width=0.1\textwidth]{ating12}+\includegraphics[width=0.1\textwidth]{ating21})=\\
&=-\lambda_j\lambda_k+i(\includegraphics[width=0.1\textwidth]{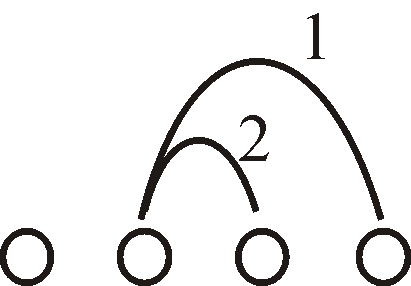}+\includegraphics[width=0.1\textwidth]{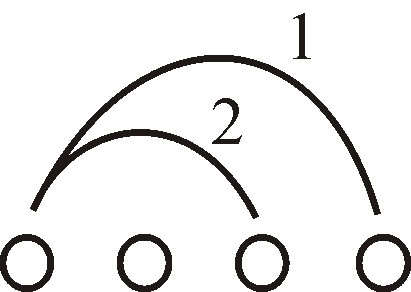}-2\includegraphics[width=0.1\textwidth]{epsk})=0,
\end{align*}
\noindent while the second can be rewritten as:
\begin{align*}
&(\includegraphics[width=0.1\textwidth]{gplus2}-i\includegraphics[width=0.1\textwidth]{glus3})(\includegraphics[width=0.1\textwidth]{gplus1}-i\includegraphics[width=0.1\textwidth]{gplus4})=\\
&=\includegraphics[width=0.1\textwidth]{intersect}-\includegraphics[width=0.1\textwidth]{inside}-i(\includegraphics[width=0.1\textwidth]{ating12a}+\includegraphics[width=0.1\textwidth]{ating21a})=\\
&=\includegraphics[width=0.1\textwidth]{separat}-i(\includegraphics[width=0.1\textwidth]{ating12a}+\includegraphics[width=0.1\textwidth]{ating21a})=\\
&=(i\includegraphics[width=0.1\textwidth]{epsj})(-i\includegraphics[width=0.1\textwidth]{epsk})-i(\includegraphics[width=0.1\textwidth]{ating12a}+\includegraphics[width=0.1\textwidth]{ating21a})=\\
&=-\lambda_j\lambda_k-i(\includegraphics[width=0.1\textwidth]{ating12a}+\includegraphics[width=0.1\textwidth]{ating21a})=0.
\end{align*}
\noindent Summing the two up, we obtain the following equation:
\begin{align*}
&-2\lambda_j\lambda_k-2i\includegraphics[width=0.1\textwidth]{epsk}=-2\lambda_j\lambda_k-2\lambda_k=0,
\end{align*}
hence the same polynomials in
$\lambda$:\begin{equation*}\lambda_k(\lambda_j+1)=0.\end{equation*}

\section{Primitive ideals}

\noindent As previously mentioned, primitive ideals of the universal
enveloping algebra $U(A)$, where $A$ is the simple $n$-Lie algebra,
correspond to the primitive ideals of the universal enveloping
algebra of the basic Lie algebra $\wedge^{n-1}A$ which include the
two-sided ideal $Q(A)$. Moreover, these primitive ideals are the
annihilators of highest weight, irreducible modules of
$\wedge^{n-1}A$.\\
\\
\noindent Let $Z(\lambda)$ be such a module, and denote by
$Ann(Z(\lambda))$ its annihilator. For $Ann(Z(\lambda))$ to be a
primitive ideal of $U(A)$, we would need that $Q(A)\subseteq
Ann(Z(\lambda))$, i.e. $Q(A)$ acts trivially on $Z(\lambda)$. Hence,
in order to find the primitive ideals of $U(A)$, we want to find
those weights $\lambda$, such that the annihilator of the
irreducible, highest weight module of weight $\lambda$ includes
$Q(A)$. These are precisely the weights we have determined in the
previous sections. Hence, we have already proved the following
theorem.
\begin{thm}
Let $I$ be a primitive ideal of $U$, the universal enveloping
algebra of the basic Lie algebra $\wedge^{n-1}A$. $I$ is a primitive
ideal of $U(A)$, the universal enveloping algebra of $A$, if and
only if $I$ is the annihilator of an irreducible, highest weight
module of $A$ with highest weight $\lambda$ of the form mentioned in
either Theorem \ref{j6} or Theorem \ref{j7}, depending on the parity
of $n+1$.%satisfying the following condition: there exist
%$t\in\{1,\ldots,\lfloor\frac{n+1}{2}\rfloor\}$ such that
%$\lambda_1=\ldots=\lambda_{t-1}=-1$, $\lambda_t=x\in\mathbb{C}$ and
%$\lambda_{t+1}=\ldots=\lambda_{\lfloor\frac{n+1}{2}\rfloor}=0$.
\end{thm}
\noindent As an example, we will show that the Joseph ideal is a
primitive ideal of $U(A)$, the enveloping algebra of the simple
$n$-Lie algebra $A$. To avoid complications we fix $n>4$. \\
\\
\noindent In \cite{Joseph}, Joseph constructed a primitive,
completely prime ideal in $U(so(n+1))$ corresponding to the closure
of the minimal nilpotent orbit of the coadjoint action; and computed
its infinitesimal character. We will denote this ideal by $J$. To
show that $J$ is a primitive ideal of $U(A)$, we must prove that
$Q(A)\subseteq J$. Denote by $\alpha$ the highest root of $so(n+1)$
and recall that $S^2(so(n+1))=V(2\alpha)\oplus V(0)\oplus W$. In
\cite{GS} it was shown that the Joseph ideal $J$ is equal to the
ideal generated by $W$ and $C-c_0$, where $c_0$ is the eigenvalue
for the
Casimir $C$ for the infinitesimal character that Joseph obtained.\\
\\
\noindent Above, we have proven that $R\subseteq S^2(so(n+1))$ and
that $S^2(so(n+1))$ decomposes as $R\oplus \textrm{Ker}\phi$, where
$\phi:S^2(so(n+1))\to\wedge^4 A$ is defined on monomials as
\[(v_i\wedge v_j)\odot(v_k\wedge v_l)\mapsto v_i\wedge v_j\wedge v_k\wedge
v_l.\] Hence, in order to show $Q(A)\subseteq J$, we must prove that
$R\cap (V(2\alpha)\oplus V(0))=\{0\}$, or equivalently that
$V(2\alpha)\oplus V(0)\subseteq \textrm{Ker}\phi$. Taking into
account that the highest root of the Lie algebra $so(n+1)$ is given
by $\epsilon_1+\epsilon_2$ and that $V(2\alpha)$ is generated by
$(v_1\wedge v_2)\odot(v_1\wedge v_2)$, this follows directly. Thus,
we may conclude that the Joseph ideal $J$ is indeed a primitive
ideal of $U(A)$.\\

\bibliographystyle{amsplain}
\def\lllll{}

\appendix
\section{Another method for the smallest case}

\noindent  For the case $n+1=4$ we will give another method for
obtaining the possible highest weights which transform
$Z(\lambda)=U/J(\lambda)$ into an irreducible representation of $A$.
If $n+1=4$ it follows that $A$ is the simple 3-Lie algebra and its
basic Lie algebra is $so(4)$. Unlike before, $so(4)$ is not a simple
Lie algebra, but it is isomorphic to $sl(2)\oplus sl(2)$. Our
approach will be to find the two $sl(2)$'s sitting inside $so(4)$
and compute the generator of $Q(A)$ in terms of the elements of
these two simple Lie algebras.\footnote{Notice that in this case
$Q(A)$ will have only one generator.}
We will then impose the condition that $Q(A)$ acts trivially on the highest weight module $Z(\lambda)$.\\
\\
Let $\bar{x}_1=v_1\wedge v_{-2}, \bar{y}_1=v_{-1}\wedge v_2$ and
$\bar{x}_2=v_1\wedge v_2, \bar{y}_2=v_{-1}\wedge v_{-2}.$ Then
$[\bar{x}_1,\bar{y}_1]=-\epsilon_1+\epsilon_2=:\bar{h}_1$ and
$[\bar{x}_2,\bar{y}_2]=-\epsilon_1-\epsilon_2=:\bar{h}_2$, while
$[\bar{x}_1,\bar{y}_2]=[\bar{x}_2,\bar{y}_1]=0$. If we compute the
rest of the commutators we obtain:
\[[\bar{h}_1,\bar{x}_1]=-2\bar{x}_1\textrm{  }[\bar{h}_1,\bar{y}_1]=2\bar{y}_1\]
\[[\bar{h}_2,\bar{x}_2]=-2\bar{x}_2\textrm{  }[\bar{h}_2,\bar{y}_2]=2\bar{y}_2\]
\[[\bar{h}_1,\bar{x}_2]=[\bar{h}_1,\bar{y}_2]=[\bar{h}_2,\bar{x}_1]=[\bar{h}_2,\bar{y}_1]=0.\]
Thus, if we define
\[x_j=i\bar{x}_j\textrm{,  }y_j=i\bar{y}_j\textrm{ and }h_j=-\bar{h}_j, \forall j\in \{1,2\}
\] we obtain two $sl(2)$'s, namely $\{x_1,h_1,y_1\}$ and
$\{x_2,h_2,y_2\}$.\\
\\
\noindent In $so(4)$ the only generator of the two-sided ideal
$Q(A)$ is the relation:
\[X=e^{12}e^{34}+e^{14}e^{23}-e^{13}e^{24}.\]
Since $h_1=\epsilon_1-\epsilon_2$ and $h_2=\epsilon_1+\epsilon_2$ it
follows that \[e^{12}=\frac{h_1+h_2}{2i}\textrm{ and
}e^{34}=\frac{h_2-h_1}{2i}.\] Next, we compute the rest of the
products in the relation $X$. We obtain
\[e^{23}=\frac{1}{2}(x_1+x_2-y_1-y_2)\]
\[e^{14}=\frac{1}{2}(-x_1+x_2+y_1-y_2)\]
\[e^{13}=\frac{1}{2i}(x_1+x_2+y_1+y_2)\]
\[e^{24}=\frac{-1}{2i}(-x_1+x_2-y_1+y_2),\] which in
turn gives
us:\[e^{14}e^{23}-e^{13}e^{24}=\frac{h_1+2y_1x_1-h_2-2y_2x_2}{2}.\]
The relation $X$ becomes then:
\[X=\frac{h_1^2-h_2^2}{4}+\frac{h_1-h_2}{2}+y_1x_1-y_2x_2.\]

\noindent Denote by $V(\mu_1,\mu_2)=<\mathbb{1}>$ some highest
weight module of $so(4)$. For this module to be a module for the
simple 3-Lie algebra $A$ we want the relation $X$ to act trivially
on $V(\mu_1,\mu_2)$. Since $h_j\mathbb{1}=\mu_j\mathbb{1}$ for all
$j\in\{1,2\}$ and $x_j\mathbb{1}=0$, we can conclude that $X$ acts
on $V(\mu_1,\mu_2)$ as:\[\mu_1^2-\mu_2^2+2\mu_1-2\mu_2.\] We impose
the condition that
\[\mu_1^2-\mu_2^2+2\mu_1-2\mu_2=0\Leftrightarrow
(\mu_1+1)^2=(\mu_2+1)^2.\] Hence, we have obtained two solutions,
namely \[\mu_1=\mu_2\textrm{ and }\mu_1+\mu_2=-2.\]

\noindent Thus, $V(\mu_1,\mu_2)=V(\mu_1)\otimes V(\mu_2)$ is a 3-Lie
algebra module if either both weights coincide, or their sum is -2.\\
\\
\noindent There still remains the matter of determining $\lambda_1$
and $\lambda_2$ (in the notation previously used). The formulas for
$h_1$ and $h_2$ tell us that $\mu_1=\lambda_1-\lambda_2$ while
$\mu_2=\lambda_1+\lambda_2$. If $\mu_1=\mu_2$ we obtain that
$\lambda_2=0$ and if $\mu_1+\mu_2=-2$ we obtain the solution
$\lambda_1=-1$. We observe that these are also the solutions of the
polynomial: $\lambda_2(\lambda_1+1)=0$.

\end{document}